\documentclass[11pt,a4paper,reqno]{amsart}

\usepackage{amsmath}
\usepackage{amssymb}
\usepackage{amsthm}
\usepackage[shortlabels]{enumitem}
\usepackage[british]{babel}
\usepackage{microtype}
\usepackage{tikz}
\usepackage{xcolor}
\usepackage{hyperref}
\usepackage{dsfont}

\calclayout
\setlist[enumerate]{leftmargin=2.2em,itemsep=2pt,topsep=4pt}
\hypersetup{
 colorlinks=true,
 linkcolor=blue!45!black,
 citecolor=blue!45!black,
 urlcolor=blue!45!black,
 pdftitle={A fast dynamo with zero topological entropy},
 pdfauthor={Lukas Niebel},
}

\theoremstyle{plain}
\newtheorem{theorem}{Theorem}[section]
\newtheorem{lemma}[theorem]{Lemma}
\newtheorem{proposition}[theorem]{Proposition}
\newtheorem{corollary}[theorem]{Corollary}
\newtheorem{conjecture}[theorem]{Conjecture}
\theoremstyle{remark}
\newtheorem{remark}[theorem]{Remark}
\numberwithin{equation}{section}

\newcommand{\T}{\mathbb{T}}
\newcommand{\R}{\mathbb{R}}
\newcommand{\CC}{\mathbb{C}}
\newcommand{\Z}{\mathbb{Z}}
\newcommand{\Lp}{\mathrm{L}}
\newcommand{\C}{\mathrm{C}}
\newcommand{\Hs}{\mathrm{H}}
\newcommand{\W}{\mathrm{W}}
\newcommand{\cC}{\mathcal{C}}
\newcommand{\cG}{\mathcal{G}}
\newcommand{\cI}{\mathcal{I}}
\newcommand{\cL}{\mathcal{L}}

\newcommand{\cR}{\mathcal{R}}

\newcommand{\cT}{\mathcal{T}}
\newcommand{\bbA}{\mathbb{A}}
\newcommand{\frl}{\mathfrak{l}}
\newcommand{\eps}{\varepsilon}
\newcommand{\dd}{\,\mathrm{d}}
\DeclareMathOperator{\supp}{supp}

\DeclareMathOperator{\curl}{curl}
\DeclareMathOperator{\Div}{div}
\renewcommand{\Re}{\operatorname{Re}}
\renewcommand{\Im}{\operatorname{Im}}
\DeclareMathOperator{\Lip}{Lip}
\newcommand{\id}{\mathrm{Id}}
\newcommand{\inn}{\mathrm{in}}
\newcommand{\bg}{\mathrm{bg}}
\newcommand{\loc}{\mathrm{loc}}
\newcommand{\axi}{\mathrm{ax}}
\newcommand{\LL}{\mathrm{L}}
\newcommand{\e}{\mathrm{e}}
\newcommand{\one}{\mathds{1}}
\newcommand{\inner}[2]{\left\langle#1,#2\right\rangle}
\newcommand{\norm}[1]{\left\|#1\right\|}
\newcommand{\abs}[1]{\left\lvert#1\right\rvert}
\newcommand{\jump}[1]{\left[\!\left[#1\right]\!\right]}

\title[A fast dynamo with zero topological entropy]{A fast dynamo with zero topological entropy}
\author{Lukas Niebel}
\address[Lukas Niebel]{ETH Z\"urich, Department of Mathematics, R\"amistrasse 101, 8092 Z\"urich, Switzerland.}
\email{lukas.niebel@math.ethz.ch}
\date{\today}


\begin{document}

\begin{abstract}
  We construct an autonomous Lipschitz fast dynamo on the flat three-torus
  whose particle flow has zero topological entropy and whose ideal induction
  equation has no exponential growth.
  We perform a multiscale construction starting from the Ponomarenko flow.
  For every sufficiently small positive
  magnetic diffusivity, the induction operator on $\mathbb{T}^3$ has an eigenvalue with real
  part bounded below by a positive constant independent of the diffusivity.
  A corresponding non-zero real-valued magnetic field satisfies an exact
  exponential growth law in $\mathrm{L}^2$. For the same velocity field, every
  solution of the ideal problem grows at most linearly in $\mathrm{L}^2$. Moreover, the particle
  flow and its inverse have Lipschitz constants growing at most linearly
  in time and thus zero topological entropy. The velocity field is real-valued, divergence-free, differentiable
  everywhere and smooth away from one circle, but is not $\mathrm{C}^1$.
\end{abstract}

\maketitle

\section{Introduction and main result}
\label{sec:introduction}

\subsection{The induction equation and the fast-dynamo problem}
To start with we fix some notations and terminologies. We work on the flat unit torus
$\T^3=\R^3/\Z^3 $, equipped with the Lebesgue measure, and write
$x=(y,x_3) $ with $y=(x_1,x_2)\in\T^2 $. We call the
$y $-directions horizontal and the $x_3 $-direction axial, and write
$e_3=(0,0,1) $. The subscript $\perp $ denotes horizontal components
or differential operators acting in the horizontal variables. We write
$\norm{\cdot}_2 $ for the $\Lp^2 $-norm. All complex inner products are linear in their first argument.
We omit the domain from Sobolev spaces and norms whenever it is clear from
context.

Given a divergence-free
velocity field $u $ and a magnetic diffusivity $\eps>0 $, the magnetic
field $B $ solves
\begin{equation}
  \begin{cases}
    \partial_tB=\curl(u\times B)+\eps\Delta B, \\
    \Div B=0,                                  \\
    B|_{t=0}=B_{\inn}.
  \end{cases}
  \label{eq:induction}
\end{equation}
This is the \emph{kinematic} dynamo problem. The identity
\[
  \curl(u\times B)
  =u\,\Div B-B\,\Div u+(B\cdot\nabla)u-(u\cdot\nabla)B
\]
shows that, if $u $ and $B $ are divergence-free, the first equation is
equivalent to
\[
  \partial_tB+(u\cdot\nabla)B-(B\cdot\nabla)u=\eps\Delta B.
\]
The transport term moves the magnetic field, the stretching term
$(B\cdot\nabla)u $ may amplify it, and diffusion damps the small scales
created by stretching. Formally, this competition is expressed by the
energy identity
\begin{equation}   \label{eq:energy-competition}
  \frac12\frac{\dd}{\dd t}\norm{B(t)}_2^2
  =
  \int_{\T^3}B\cdot(B\cdot\nabla)u\dd x
  -\eps\norm{\nabla B}_2^2.
\end{equation}

For $\eps>0$ and a time-independent, divergence-free velocity field
$u\in\W^{1,\infty}(\T^3;\R^3)$, let
\begin{equation}
  L_{\eps,u}B
  := \curl(u\times B)+\eps\Delta B
  \label{eq:physical-generator}
\end{equation}
on the complex Hilbert space
\[
  \Lp^2_\sigma(\T^3)
  :=
  \{B\in\Lp^2(\T^3;\CC^3)\, : \,\Div B=0\},
\]
with domain
$\Hs^2(\T^3;\CC^3)\cap\Lp^2_\sigma(\T^3) $.
This operator generates an analytic $\C_0$-semigroup on
$\Lp^2_\sigma$. More properties of the semigroup are collected in
Appendix \ref{sec:semigroup-growth}.

The coefficient $\eps $ is the magnetic diffusion coefficient. For a fixed
magnetic permeability, it is inversely proportional to electrical
conductivity. The limit $\eps\downarrow0 $ is the high-conductivity regime. We call the problem
\emph{ideal} when $\eps=0 $ and \emph{resistive} if $\eps>0 $.

The key question on slow and fast dynamos, introduced by Vainshtein
and Zel'dovich \cite{VZ72}, asks whether the exponential
growth rate survives this high-conductivity limit. For a fixed autonomous velocity field, we define the
maximal growth rate by
\begin{equation}
  \gamma(u,\eps)
  :=
  \lim_{t\to\infty}\frac1t
  \log\norm{\e^{tL_{\eps,u}}}_{\Lp^2_\sigma\to\Lp^2_\sigma}.
  \label{eq:semigroup-growth-rate}
\end{equation}
The limit exists because the logarithm of the semigroup norm is
subadditive and every $\C_0 $-semigroup is locally bounded in time
\cite[Proposition~IV.2.2]{EngelNagel}.
This quantity measures only the long-time exponential growth.

We call $u $ a dynamo at diffusivity $\eps $ when
$\gamma(u,\eps)>0 $. It is a \emph{fast dynamo} if
\[
  \liminf_{\eps\downarrow0}\gamma(u,\eps)>0.
\]
If $\gamma(u,\eps)>0 $ for every sufficiently small $\eps $, but
$\gamma(u,\eps)\to0 $, we call it a \emph{slow dynamo}.

For every $\eps>0$, this growth rate is determined by the spectrum of $L_{\eps,u}$:
\[
  \gamma(u,\eps)
  =\sup\{\Re\lambda\, : \,\lambda\in\sigma(L_{\eps,u})\}.
\]
The supremum is attained at an eigenvalue;
see Lemma \ref{lem:semigroup-spectral-growth}.

Thus the fast-dynamo question for an autonomous velocity field is an
eigenvalue problem. On the flat three-torus, the conjecture recorded as
Problem~1994--28 in Arnold's collection \cite{Arnold04} has the following
form; see also \cite[Chapter~V]{AK98}. It was open when the first version
of this manuscript appeared and has since been resolved; see Remark \ref{rem:subsequent-resolution}.

\begin{conjecture}[Smooth autonomous fast dynamo]
  \label{conj:smooth-fast-dynamo}
  There exists a single real-valued, time-independent velocity field
  \[
    u\in\C^\infty(\T^3;\R^3),\qquad \Div u=0,
  \]
  such that
  \[
    \liminf_{\eps\downarrow0}\gamma(u,\eps)>0.
  \]
\end{conjecture}

In other words, the conjecture asks for one smooth velocity field $u $ and constants
$\eps_0,\gamma_0>0 $, such that
$L_{\eps,u} $ has an eigenvalue in
$\{\Re z\ge\gamma_0\} $ for every
$0<\eps\le\eps_0 $. Here and below, a \emph{magnetic eigenmode} is a
non-zero divergence-free eigenfunction $V $ of the relevant induction
operator; if its eigenvalue is $\lambda $, it generates the complex solution
$t\mapsto\e^{t\lambda}V $. The eigenvalue and a corresponding magnetic
eigenmode may depend on $\eps $.

\subsection{Main result}

For a Lipschitz velocity field,
we denote its \emph{particle flow} by $\Phi_t$, where
\[
  \frac{\dd}{\dd t}\Phi_t(x)=u(\Phi_t(x)),
  \qquad \Phi_0(x)=x.
\]
We write $\Lip(F)$ for the optimal Lipschitz constant of a
map $F$ and $h_{\mathrm{top}}(F)$ for the topological entropy of a
continuous self-map of a compact metric space. We refer to
\cite[Section~2.5e]{HasselblattKatok02} for its definition.
Moreover, $\cT_0(t)$
denotes the solution group on $\Lp^2_\sigma(\T^3)$ of the ideal induction equation.

\begin{theorem}[A fast dynamo with zero topological entropy]\label{thm:main}
  There exist a real-valued autonomous velocity field
  \[
    u\in \W^{1,\infty}(\T^3;\R^3),\qquad \Div u=0,
  \]
  and constants $\eps_0>0$, $\gamma_0>0$, such that $\gamma(u,\eps)\ge\gamma_0$ for every $0<\eps\le\eps_0$. In particular,
  \begin{equation}
    \liminf_{\eps\downarrow0}\gamma(u,\eps)\ge\gamma_0.
    \label{eq:main-growth-conclusion}
  \end{equation}
  More precisely, for every $0<\eps\le\eps_0 $, there are
  \[
    K_\eps\in\Z\setminus\{0\},\qquad
    \lambda_\eps\in\CC,\qquad
    0\neq b^\eps\in\Hs^2(\T^2;\CC^3),
  \]
  such that the complex magnetic eigenmode
  \begin{equation}
    V^\eps(y,x_3):=\e^{2\pi iK_\eps x_3}b^\eps(y)
    \in\Hs^2(\T^3;\CC^3)\cap\Lp^2_\sigma(\T^3),
    \quad
    L_{\eps,u}V^\eps=\lambda_\eps V^\eps,
    \quad
    \Re\lambda_\eps\ge\gamma_0.
    \label{eq:main-spectral-conclusion}
  \end{equation}
  Moreover, the real-valued datum
  \[
    B_{\inn}^\eps:=\Re V^\eps
    \in\Hs^2(\T^3;\R^3)\cap\Lp^2_\sigma(\T^3)
  \]
  is non-zero, and its solution satisfies
  \begin{equation}\label{eq:main-thm-growth}
    \norm{B^\eps(t)}_{\Lp^2(\T^3)}
    =
    \e^{t\Re\lambda_\eps}\norm{B_{\inn}^\eps}_{\Lp^2(\T^3)}
    \qquad\text{for every }t\ge0.
  \end{equation}

  Moreover, $h_{\mathrm{top}}(\Phi_t)=0$ for all $t\in\R$ and
  \[
    \lim_{t\to\infty}\frac1t
    \log\norm{\cT_0(t)}_{\Lp^2_\sigma\to\Lp^2_\sigma}=0.
  \]
\end{theorem}

\begin{remark}
  \label{rem:subsequent-resolution}
  After the first version of the present manuscript appeared on
  arXiv, Conjecture \ref{conj:smooth-fast-dynamo} was resolved in
  \cite[Theorem~1.1]{CSV26autonomous}. Their construction gives a
  rigorous realisation of the stretch--fold--shear mechanism proposed
  in the physics literature \cite{BaylyChildress88,ChildressGilbert95}.
  They first construct smooth time-periodic fast dynamos and then realise
  the period dynamics in a smooth autonomous flow. In both the time-periodic and the autonomous constructions,
  they isolate an exponentially growing distributional eigenmode of
  the ideal equation on an anisotropic Banach space and show that the
  spectral instability persists under the singular perturbation
  $\eps\Delta$; see \cite[Sections~1.1--1.4]{CSV26autonomous}.

  Let us compare the behaviour without diffusion. The smooth fast
  dynamos in \cite{CSV26autonomous} have positive topological entropy
  and positive exponential growth of the ideal solution group in
  $\Lp^2_\sigma$ operator norm. For the relation between entropy,
  particle stretching and ideal growth, we refer to
  \cite{KlapperYoung95}. For the velocity constructed here, the
  Lipschitz constants of the flow maps and the operator norm of the
  ideal solution group grow at most linearly; see Corollary \ref{cor:ideal-growth}. At Lipschitz
  regularity, fast dynamo action is compatible with zero topological entropy and
  zero ideal exponential growth.

  This raises the question whether Lipschitz regularity is optimal
  for fast dynamo action with zero topological entropy. In particular,
  does there exist a divergence-free autonomous velocity
  $u\in\C^1(\T^3;\R^3)$ which is a fast dynamo and whose time-one
  map satisfies $h_{\mathrm{top}}(\Phi_1)=0$?
\end{remark}

\begin{remark}
  The magnetic eigenmode $V^\eps $ is complex. If
  $\Im\lambda_\eps\neq0 $, its real part is not itself a magnetic eigenmode,
  since it is not an eigenfunction of $L_{\eps,u} $. We obtain the
  real-valued growth law in Lemma \ref{lem:real-axial-mode}.
\end{remark}

\begin{remark}
  \label{rem:scope-intro}
  The velocity and the positive lower bound in Theorem \ref{thm:main}
  are independent of $\eps$. However, the magnetic eigenmode and the
  real initial datum may depend on $\eps$. We do not claim simplicity
  or continuity of $\lambda_\eps$, that it attains the spectral bound,
  or the existence of a common growing initial datum for all diffusivities.
  Since \(K_\eps\neq0\), every magnetic field constructed here has zero
  spatial mean.
  The velocity is differentiable everywhere and smooth away from one axial
  circle. Its derivative vanishes on that circle but is discontinuous
  there, so the velocity is not $\C^1 $. Details can be found in
  Subsection \ref{sec:lipschitz-obstruction}.
\end{remark}

\begin{remark}
  We may normalise the size of the velocity field. Put
  $M=\max\{1,\norm{u}_{\W^{1,\infty}}\} $ and
  $\widetilde u=u/M $. Then
  \[
    L_{\eps,\widetilde u}=M^{-1}L_{M\eps,u},
    \qquad
    \gamma(\widetilde u,\eps)=M^{-1}\gamma(u,M\eps).
  \]
  Hence Theorem \ref{thm:main} remains valid for $\widetilde u $, with
  threshold $\widetilde \eps = \eps_0/M $ and lower bound
  $\widetilde \gamma = \gamma_0/M $ on the growth rate. Moreover, the particle flow and ideal induction
  group for $\widetilde u $ are the corresponding objects for $u $
  evaluated at time $t/M $. This preserves the ideal growth rate and
  the vanishing of topological entropy. In particular, the velocity may be required to satisfy
  $\norm{\widetilde u}_{\W^{1,\infty}}\le1 $ without changing any of the conclusions of Theorem \ref{thm:main}.
\end{remark}

\begin{remark}
  For each fixed $B_{\inn}\in\Lp^2_\sigma$,
  $\e^{tL_{\eps,u}}B_{\inn}$ converges strongly in $\Lp^2$ to
  $\cT_0(t)B_{\inn}$ as $\eps\downarrow0$, uniformly on
  bounded time intervals in $[0,\infty)$. This follows from the
  Trotter--Kato approximation theorem
  \cite[Theorem~III.4.8]{EngelNagel} and the Friedrichs commutator
  lemma \cite[Lemma~II.1]{DiPernaLions89}.
\end{remark}

\subsection{Construction and proof strategy}
\label{sec:ideas}

In this subsection, we outline the construction in Theorem \ref{thm:main}.
The proof is a multiscale construction starting from the Ponomarenko flow.
We seek a velocity field $u=(u_\perp,u_3)$ on $\T^3$, independent of $x_3$, and a magnetic eigenmode $V$
of the form
\[
  V(y,x_3)=\e^{2\pi iKx_3}(h(y),b_3(y)),
  \qquad K\in\Z\setminus\{0\}.
\]
The horizontal equations are closed and thus we may reduce to the following
horizontal spectral problem on $\T^2$:
\[
  \lambda h=A_{\eps,K}h
  := \eps\left(\Delta_\perp-(2\pi K)^2\right)h
  -u_\perp\cdot\nabla_\perp h
  -2\pi iKu_3h
  +(h\cdot\nabla_\perp)u_\perp.
\]
Here, we do not need to impose a divergence constraint on $h$ as the axial
component $b_3$ is recovered from
$b_3=-(2\pi iK)^{-1}\Div_\perp h$ to fulfill the divergence constraint
for $V$. Our task is to construct one velocity field for which, at every
sufficiently small diffusivity, some non-zero axial Fourier mode
$K=K(\eps)$ admits an eigenvalue whose real part is bounded below by
a positive constant independent of $\eps$.

\smallskip
\noindent\emph{The reference model on $\R^2$.}
We first consider a reference spectral problem on $\R^2$.
We start from the Ponomarenko flow \cite{Ponomarenko73}. We
regularise the discontinuous interface and obtain a smooth, compactly
supported screw flow $v\colon\R^2\to\R^3$, whose particles rotate and
translate at fixed radius. We call this the reference velocity profile on $\R^2$. We denote by $\cL_\delta$ this corresponding
reference operator on $\R^2$, given by the same formulas as
$A_{\eps,K}$ with $u=v$, $\eps=\delta$ and $K=1$.
There are an interval $I=[\delta_-,\delta_+]$ containing $1$ and a fixed
reference contour $\Gamma\subset\CC$ enclosing spectrum of $\cL_\delta$
for every $\delta\in I$. The closed interior of $\Gamma$ lies strictly
in the right half-plane of $\CC$.

In the following we call the model on \(\mathbb{R}^2\) with the reference velocity profile the \emph{reference model}, and the horizontal problem on \(\mathbb{T}^2\) associated with the velocity profile \(u\) the \emph{horizontal problem}. We use the same terminology for the corresponding operators, resolvents and Riesz projections.

\smallskip
\noindent\emph{The multiscale construction.}
Diffusion acts at a rate $\eps/\ell^2$ on fields varying over a length
$\ell$. Thus an instability can remain effective as
$\eps\downarrow0$ if the magnetic length decreases like $\sqrt{\eps}$.
We therefore assemble rescaled copies of the reference profile from $v$ at all small scales.
Here, we choose integers $K_n$ growing in geometric way and we set
$\ell_n=K_n^{-1}$. In the $K_n$-th Fourier mode, the effective diffusivity is
$\delta_n=\eps/\ell_n^2$. Following the interval assignment in
\cite[Section~4.1]{CSV25}, we choose the consecutive ratios
$K_{n+1}/K_n$ uniformly greater than $1$ but at the same time sufficiently close to $1$
so that the intervals $\ell_n^2I$ cover $(0,\eps_0]$. For each
$0<\eps\le\eps_0$, we choose a cell with $\delta_n\in I$ and call it the selected cell.

The main difficulty in this construction is to preserve the instability of the model on $\R^2$ for the selected
cell in the presence of all the other scales. The velocity field being defined on disjoint supports
does not decouple the magnetic equation, since diffusion connects the
cells. To control the other scales, we add an axial velocity that is
constant near each core.
Now, in the horizontal equations, the axial
velocity enters only through the imaginary multiplier $-2\pi iKu_3$.
This helps us distinguish the cells in the spectral estimates.

Let us now define $u$. We set $R_n=\sqrt{\ell_n}$ and place the rescaled reference profiles on $\T^2$ in disjoint cells
$D(p_n,R_n)\subset\T^2$ accumulating at one point. With that the velocity field on $\T^3$ is
defined as
\begin{align*}
  u(y,x_3)={}\sum_{n\ge0}\left[ \ell_n v\left(\frac{y-p_n}{\ell_n}\right)+\frac{\tau\ell_n}{2\pi}\vartheta\left(\frac{|y-p_n|}{R_n}\right){e_3} \right].
\end{align*}
Here, $\vartheta$ is a smooth cut-off and $\tau>0$ is a parameter that will be chosen
at a later point. The added axial velocity is constant near each core and is
then cut off close to the boundary of the cell. Between the core, of scale
$\ell_n$, and the cut-off we put a buffer of width comparable to
$R_n$. In this buffer the velocity is constant. The resulting velocity field on $\T^3$
is Lipschitz and differentiable everywhere, but not $\C^1$.

\smallskip
\noindent\emph{Preserving the unstable spectrum.}
In the axial Fourier mode with index $K_n$ corresponding to the selected cell, the added axial velocity profiles on $\T^2$ contribute
$-i\tau$ on the selected core and $-i\tau K_n/K_m$ on core $m$.
Extending the selected constant axial velocity to the whole plane
and using the change of variables $Y=(y-p_n)/\ell_n$, we obtain the shifted reference operator on $\R^2$
\[
  \cL_{\delta_n}-i\tau.
\]
There is no change of time scale, so the growth rate of the shifted reference operator on $\R^2$ is the same. We compare resolvents on the shifted contour
$\Gamma_\tau=\Gamma-i\tau$ and show that the enclosed Riesz projection of the horizontal operator on $\T^2$
is non-zero.

Next, we remove the rescaled reference profile of \(v\) in the selected cell of $\T^2$ from the velocity field on $\T^2$, retaining all added axial velocity profiles, and denote the resulting background operator on \(\mathbb{T}^2\) by \(A^{\bg}_{\varepsilon,n}\). A weighted estimate for the background resolvent
equation on $\T^2$ uses the different axial shifts to control the velocity terms of the remaining
terms. Choosing $\tau$ large and then $\ell_0$ small,
we obtain, uniformly in the selected cell and diffusivity and for
$z$ on and inside $\Gamma_\tau$, the resolvent estimate
\[
  \|h\|_2+\sqrt\eps\,\|\nabla_\perp h\|_2\le C\|f\|_2,
  \qquad (z-A^{\bg}_{\eps,n})h=f.
\]

Next, we combine the resolvent on $\R^2$ for the selected cell with the background
resolvent on $\T^2$, using cut-offs on both $f$ and $h$. The cut-offs vary in the
constant-coefficient buffer, so the errors come only from diffusion.
The corresponding approximate inverse for the horizontal problem on $\T^2$ has error
\[
  \|E_n(z)\|_{2\to2}
  \le C\left(\frac{\sqrt\eps}{R_n}+\frac\eps{R_n^2}\right)
  \le C\left(\sqrt{\ell_n}+\ell_n\right).
\]
Taking $\ell_0$ sufficiently small, we obtain the resolvent of the horizontal operator on $\T^2$ along
$\Gamma_\tau$ by a Neumann series.

It remains to show that there is spectrum inside the shifted contour. The reference
model on $\R^2$ supplies a compactly supported unit test function $f_*$
with $\Re\langle P_\delta f_*,f_*\rangle>1/2$ throughout $I$, where
$P_\delta$ is the Riesz projection of the reference operator on $\R^2$. Testing the Riesz
projection of the horizontal operator on $\T^2$ against the normalised rescaling of $f_*$ in the selected
cell reproduces this inner product up to a small error. Hence this
projection is non-zero and gives an eigenvalue whose real part
is bounded below independently of $\eps$. These steps are carried out
in Sections~\ref{sec:localization}--\ref{sec:parametrix}.

Finally, we recover $b_3$ to obtain a magnetic eigenmode of the induction operator on $\T^3$ and take the real part
of its evolution. Integration in the non-zero axial Fourier mode removes
the time-dependent complex phase, giving the exact
exponential growth law; see Section \ref{sec:finish}.

\smallskip
\noindent\emph{The ideal dynamics.}
The velocity field $u$ of this work has no exponential particle stretching. Within
each cell, its flow has the form
\[
  (r,\theta,x_3)\longmapsto
  \left(r,\theta+t\omega_n(r),x_3+t g_n(r)\right),
\]
where $\omega_n(r)$ and $g_n(r)$ are, respectively, the angular
and axial velocities at radius $r$ in the $n$-th cell.
Since $r|\omega_n'|$ and $|g_n'|$ are uniformly bounded, the flow and
its inverse have Lipschitz constants bounded by $C(1+|t|)$. We deduce
zero topological entropy and, by the Cauchy formula, zero exponential
growth for the ideal induction equation. In contrast, the resistive eigenmodes vary
on scales $\ell_n\asymp\sqrt\eps$, where diffusion remains of order one.
Thus the growing mode changes with $\eps$. It does not arise by
perturbing an exponentially growing solution of the ideal problem.

\subsection{Relation to earlier work}
\label{sec:short-literature}

We first recall related works on the instability of the reference model on $\R^2$ used in our construction.
The modal reduction, transmission conditions and Bessel dispersion
relation go back to Ponomarenko \cite{Ponomarenko73}; see also
\cite[Supplemental Material, Section~II.A]{GKS18}.
Gilbert studied the high-conductivity limit of the discontinuous screw
flow \cite{Gilbert88}. Smooth screw profiles and regularised helical
vortex sheets were considered in \cite{RSS88,GVR07}. More recently,
Navarro-Fern\'andez and Villringer constructed unstable critical-layer
eigenmodes for smooth helical profiles
\cite[Theorems~1.1 and~1.2]{NV25}. Their constructed branch has a growth
rate of order $\eps^{1/3}$.

The multiscale strategy of assembling rescaled local flows in disjoint
regions of decreasing size goes back at least to the loss-of-regularity
construction of Alberti, Crippa and Mazzucato \cite[Section~3]{ACM14},
developed in detail in \cite[Section~3.1]{ACM19}; see also Jabin
\cite{Jabin16}. Related multiscale ideas underlie the anomalous-diffusion
construction of Armstrong and Vicol \cite{ArmstrongVicol25}.
A talk by Armstrong about \cite{ArmstrongVicol25} was a key inspiration for our work. In the context of the fast-dynamo
problem, Coti Zelati, Sorella and Villringer \cite[Section~4.1]{CSV25}
employed a multiscale localisation argument to construct an autonomous
Lipschitz fast dynamo on $\R^3$.
Their whole-space localisation uses an additional index for the time
horizon, with larger regions for longer times.

Smooth autonomous fast dynamos on compact curved three-manifolds were
constructed in \cite{AZRS81,CL97}. Earlier constructions on the flat
torus use time-dependent velocities: \cite{Rowan25} gives positive
limsup growth for a prescribed countable set of diffusivities,
\cite{SV25} gives positive limsup growth for every sufficiently small
diffusivity, and \cite{CSV26} gives an unstable Floquet mode with
continuous diffusion. The subsequent construction of a smooth autonomous fast dynamo
on $\T^3$ is discussed in Remark \ref{rem:subsequent-resolution}.
For the ideal equation, Coti Zelati and Navarro-Fern\'andez obtain
almost-sure exponential growth for randomised ABC flows \cite{CNF24}.
Navarro-Fern\'andez \cite[Theorem~1.1]{NF26} constructs a deterministic,
time-periodic velocity field, which is Lipschitz in space, for which every
non-zero divergence-free initial magnetic field in $\Lp^2(\T^3)$
gives a solution of the ideal induction equation with exponential growth in $\Lp^2$.
We refer to \cite{Rowan26random} for a proof of smooth random fast
dynamo action on $\T^3$.

For sufficiently smooth velocities, the
absence of exponential particle stretching excludes fast dynamo action
\cite{Vishik89,FriedlanderVishik91}. Moreover, the entropy bound of
Klapper and Young \cite{KlapperYoung95} excludes a $\C^\infty$ fast
dynamo with zero topological entropy.

\subsection{Organisation of the paper}

In Section \ref{sec:local-model}, we construct the smooth reference velocity profile on $\R^2$
and prove the spectral statement for the reference operator on $\R^2$.
Section \ref{sec:velocity} contains the multiscale velocity construction.
The assertions concerning the particle flow and ideal induction are proven in Section \ref{sec:particle}.
We introduce the horizontal operator on $\T^2$ in Section \ref{sec:axial}. In
Section \ref{sec:localization}, a weighted form estimate gives the
bound for the background resolvent on $\T^2$. We combine this resolvent with the
resolvent on $\R^2$ for the selected cell in Section \ref{sec:parametrix} and prove that the resulting
Riesz projection of the horizontal operator on $\T^2$ is non-zero. In Section \ref{sec:finish}, we recover the
axial magnetic component, prove Theorem \ref{thm:main}, and identify the
magnetic length scale. Appendix \ref{sec:semigroup-growth} establishes
the semigroup and spectral-growth statements.

\subsection*{Acknowledgements}
Lukas Niebel is funded by SNSF Starting Grant
TMSGI2\textunderscore226018 and by the Deutsche Forschungsgemeinschaft (DFG,
German Research Foundation) under Germany's Excellence Strategy
EXC 2044/2--390685587, Mathematics M\"unster:
Dynamics--Geometry--Structure.

\subsection*{Declaration of AI Use}
The author directed the mathematical research of this work.
The first version of this manuscript was assisted by OpenAI's GPT-5.5 Pro and GPT-5.6 Sol.
Further simplifications in this revised version of the manuscript were obtained with the help of
OpenAI's GPT-6 Astra. All mathematical claims, calculations, and AI-generated suggestions were
critically reviewed and verified by the author, who takes full responsibility
for the mathematical content and the manuscript.

A Lean formalisation of this manuscript can be found in \cite{NiebelLean}.

\section{A compactly supported helical reference velocity profile on $\R^2$ with unstable spectrum}
\label{sec:local-model}

In this section, we construct a smooth compactly supported reference velocity profile on $\R^2$ with
unstable horizontal spectrum on an interval of diffusivities. The axial magnetic component will
be recovered in Section \ref{sec:finish}.

\subsection{The horizontal operator on $\R^2$ and its weak form}

We write
$y^\perp=(-y_2,y_1)$ and $w=(w_\perp,w_3)$, where
$w_\perp\colon\R^2\to\R^2$ and $w_3\colon\R^2\to\R$.
Let $w$ be a smooth velocity independent of $x_3$ that satisfies
$\Div_\perp w_\perp=0$. We consider magnetic fields whose horizontal
components have the form
\begin{equation*}
  B_\perp(t,y,x_3)=\e^{2\pi i x_3}b(t,y),
\end{equation*}
where $b$ depends only on time and the horizontal variables
$y=(x_1,x_2)$. The key observation is that the evolution equation for $b$ does not involve
the axial magnetic component.
At unit diffusivity,
its generator is
\begin{equation} \label{eq:horizontal-induction}
  A_wb=(\Delta_\perp-4\pi^2)b
  -w_\perp\cdot\nabla_\perp b-2\pi iw_3b
  +(b\cdot\nabla_\perp)w_\perp,
  \qquad b\colon\R^2\to\CC^2.
\end{equation}
We want to use this operator for bounded discontinuous velocities.
In this case, for a jump velocity, stretching can contain a distribution supported on
an interface. The weak formulation must retain this contribution.

For this purpose, we interpret the stretching term by integration by parts as
\[
  \inner{(b\cdot\nabla_\perp)w_\perp}{\phi}
  =-\int_{\R^2}w_\perp\cdot
  \left((\Div_\perp b)\overline\phi
  +(b\cdot\nabla_\perp)\overline\phi\right)\dd y.
\]
Thus no derivative of the velocity appears in the associated form
\begin{align}\label{eq:jump-sectorial-form}
  \frl_w(b,\phi)
  ={} & \inner{\nabla_\perp b}{\nabla_\perp\phi}
  +4\pi^2\inner b\phi
  +\inner{w_\perp\cdot\nabla_\perp b}\phi
  +2\pi i\inner{w_3b}\phi \nonumber              \\
      & +\int_{\R^2}w_\perp\cdot
  \left((\Div_\perp b)\overline\phi
  +(b\cdot\nabla_\perp)\overline\phi\right)\dd y
\end{align}
for $b,\phi\in\Hs^1(\R^2;\CC^2)$.
For $w\in\Lp^\infty$, this form is bounded on $\Hs^1$. Moreover,
Young's inequality gives
\begin{equation} \label{eq:sectorial-estimate-b}
  \Re\frl_w(b,b)
  \ge\frac12\norm{\nabla_\perp b}_2^2-C\norm b_2^2,
\end{equation}
where $C$ depends only on $\norm w_\infty$.
The same estimates show that, after increasing $C$ if necessary,
\[
  \abs{\Im\frl_w(b,b)}
  \le
  \Re\frl_w(b,b)+C\norm{b}_2^2
  \qquad
  \text{for every } b\in\Hs^1(\R^2;\CC^2).
\]
Consequently
$\frl_w$ is closed and sectorial. We denote by $-A_w$ the
operator on $\Lp^2(\R^2;\CC^2)$ associated with
$\frl_w$; see \cite[Theorem VI.2.1]{Kato}. In particular,
$b\in D(A_w)$ and $A_wb=f$ precisely when $b\in\Hs^1$ and
\[
  \frl_w(b,\phi)=-\inner f\phi
  \qquad\text{for every }\phi\in\Hs^1.
\]

\subsection{An unstable eigenvalue for the discontinuous flow}

We now consider the Ponomarenko screw flow \cite{Ponomarenko73}.
For an angular velocity $\Omega>0$, to be chosen below, we set the axial velocity $U=-\Omega/(2\pi)$. The relation between $\Omega $ and $U$ will become clear later on.
In polar coordinates $(r,\theta)$ with unit vectors $e_r$ and $e_\theta$, we define
\begin{equation}
  V^0(y)=\one_{\{r<1\}}\left(\Omega r e_\theta+Ue_3\right).
  \label{eq:jump-flow}
\end{equation}
The normal component vanishes on both sides of $r=1$. Hence
$\Div_\perp V^0_\perp=0$ in distributions. The separated equations,
transmission conditions and resonant dispersion relation below go back to
Ponomarenko \cite{Ponomarenko73}; see also \cite{Gilbert88} and
\cite[Supplemental Material, Section~II.A]{GKS18}.

\begin{lemma}\label{lem:jump-horizontal-mode}
  There exists $\Omega>0$ such that the operator $A_{V^0}$, with $V^0$ as in \eqref{eq:jump-flow}, has an
  eigenvalue $\lambda_*$ with $\Re\lambda_*>0$.
  An associated non-zero eigenfunction has the form
  $b(r,\theta)=\e^{i\theta}(B_r(r)e_r+B_\theta(r) e_\theta)$, is smooth on each
  side of $r=1$, and decays exponentially at infinity.
\end{lemma}

\begin{proof}
  For $j\in\Z$ we make the Ansatz
  \[
    b=\e^{ij\theta}(B_re_r+B_\theta e_\theta),
    \text{ and set } B_\pm=B_r\pm iB_\theta.
  \]
  Then, $b_1\pm ib_2=\e^{i(j\pm1)\theta}B_\pm$, where the shifted
  angular orders come from expressing the rotating polar basis in
  Cartesian coordinates. In these combinations the angular part of
  diffusion acts separately on $B_+$ and $B_-$.
  The vector Laplacian satisfies
  \[
    \e^{-ij\theta}
    \left((\Delta_\perp b)_r\pm i(\Delta_\perp b)_\theta\right)
    =
    \left(\partial_r^2+r^{-1}\partial_r
    -(j\pm1)^2r^{-2}\right)B_\pm.
  \]
  In the interior, differentiation of the polar basis in
  \(-\Omega\partial_\theta b\) contributes
  $-\Omega\e^{ij\theta}(B_re_\theta-B_\theta e_r)$.
  Stretching contributes the opposite term. Thus the complete
  transport--stretching term is $-i(j\Omega+2\pi U)b$.
  Choosing $j=1$ and recalling $2\pi U=-\Omega$, we obtain the same
  radial equations in the interior and exterior:
  \begin{equation*}
    B_\pm''+r^{-1}B_\pm'
    -\left(\nu_\pm^2r^{-2}+\kappa^2\right)B_\pm=0,
    \qquad \nu_+=2,\quad\nu_-=0,
    \qquad \kappa^2=4\pi^2+\lambda,
  \end{equation*}
  where we choose $\Re\kappa>0$.

  We use the convention $\jump f=f_{\rm out}-f_{\inn}$.
  A piecewise smooth field in $\Hs^1$ has a common trace at the
  interface. For such a field, the singular parts of the generator are
  \[
    (\Delta_\perp b)_{\rm sing}
    =\jump{\partial_rb}\,\delta_{\{r=1\}},\qquad
    \left((b\cdot\nabla_\perp)V^0_\perp\right)_{\rm sing}
    =-\Omega b_r e_\theta\,\delta_{\{r=1\}}.
  \]
  The remaining terms and the right-hand side $\lambda b$ have no interface contributions, so the singular terms must cancel. Hence the weak
  eigenvalue equation is equivalent to the bulk equations, continuity,
  and
  \begin{equation}
    \jump{B_r'}=0,\qquad
    \jump{B_\theta'}=\Omega B_r,\qquad
    \jump{B_\pm'}=\pm\frac{i\Omega}{2}(B_++B_-).
    \label{eq:transmission-horizontal}
  \end{equation}

  Let $I_\nu$ and $K_\nu$ denote the modified Bessel functions
  \cite[Section 10.25]{DLMF}. Regularity at the origin and decay at
  infinity give
  \begin{equation}
    B_\pm(r)=
    \begin{cases}
      C_\pm I_{\nu_\pm}(\kappa r)/I_{\nu_\pm}(\kappa), & 0<r<1, \\[1mm]
      C_\pm K_{\nu_\pm}(\kappa r)/K_{\nu_\pm}(\kappa), & r>1.
    \end{cases}
    \label{eq:bessel-solutions}
  \end{equation}
  We restrict to
  $\kappa=\varrho(1+it)$ with $t\in[1/2,3/4]$ and
  $\varrho>0$ sufficiently large. This segment lies in a fixed closed
  subsector of $\{\abs{\arg\kappa}<\pi/2\}$.
  The product expansion
  \cite[equation 10.40.6 and Section 10.40(iii)]{DLMF} gives
  \begin{equation}
    I_\nu(\kappa)K_\nu(\kappa)
    =\frac1{2\kappa}-\frac{4\nu^2-1}{16\kappa^3}
    +O(\varrho^{-5}),\qquad \nu\in\{0,2\},
    \label{eq:bessel-product-expansion}
  \end{equation}
  uniformly in $t$. In particular, both products are non-zero for
  sufficiently large $\varrho$, so the denominators in equation
  \eqref{eq:bessel-solutions} do not vanish.

  Applying the Wronskian identity
  \cite[equation 10.28.2]{DLMF}, we obtain
  \[
    I_\nu'K_\nu-I_\nu K_\nu'=\kappa^{-1},\qquad
    \jump{B_\pm'}=-\frac{C_\pm}{I_{\nu_\pm}(\kappa)K_{\nu_\pm}(\kappa)}.
  \]
  We set $S=C_++C_-$. The transmission conditions \eqref{eq:transmission-horizontal} are equivalent to
  \[
    C_\pm=\mp\frac{i\Omega}{2}
    I_{\nu_\pm}(\kappa)K_{\nu_\pm}(\kappa)S.
  \]
  A non-zero pair satisfies $S\neq0$. Summing these identities gives
  \begin{equation}
    \mathcal D(\kappa):=
    I_0(\kappa)K_0(\kappa)-I_2(\kappa)K_2(\kappa)
    =-\frac{2i}{\Omega}.
    \label{eq:dispersion}
  \end{equation}
  Conversely, if \eqref{eq:dispersion} holds, the preceding
  formulas with $S=1$ give coefficients whose sum is one. Hence they
  form a non-zero pair satisfying the transmission conditions.

  We now choose $\kappa$ and then $\Omega$. Applying equation
  \eqref{eq:bessel-product-expansion}, we obtain
  \[
    \varrho^3\mathcal D(\varrho(1+it))
    =(1+it)^{-3}+O(\varrho^{-2})
  \]
  uniformly for $t\in[1/2,3/4]$.
  We compute
  \[
    \Re(1+it)^{-3}=\frac{1-3t^2}{(1+t^2)^3},\qquad
    \Im(1+it)^{-3}=\frac{t^3-3t}{(1+t^2)^3}.
  \]
  The real part is positive at $t=1/2$ and negative at $t=3/4$.
  The imaginary part is strictly negative throughout $[1/2,3/4]$.
  For sufficiently large $\varrho$, the same statements hold for
  $\mathcal D(\varrho(1+it))$. Applying the intermediate value
  theorem, we find $t_\varrho\in(1/2,3/4)$ such that
  $\Re\mathcal D(\varrho(1+it_\varrho))=0$.
  We set
  \[
    \kappa=\varrho(1+it_\varrho),\qquad
    \Omega=-\frac{2}{\Im\mathcal D(\kappa)}>0,
    \qquad U=-\Omega/(2\pi).
  \]
  Then \eqref{eq:dispersion} holds exactly. Increasing
  $\varrho$ if necessary, we also obtain
  \begin{equation}
    \lambda_*:=\kappa^2-4\pi^2,\qquad
    \Re\lambda_*
    =\varrho^2(1-t_\varrho^2)-4\pi^2
    \ge\frac7{16}\varrho^2-4\pi^2>0.
    \label{eq:local-eigenvalue}
  \end{equation}

  Finally, we choose the preceding coefficients $C_\pm$ with $S=1$ and set
  $B_r=(B_++B_-)/2$ and $B_\theta=(B_+-B_-)/(2i)$.
  The Cartesian combinations of the resulting field are
  $b_1+ib_2=\e^{2i\theta}B_+$ and $b_1-ib_2=B_-$.
  Since $I_2(\kappa r)$ is $r^2$ times an analytic function of $r^2$
  and $I_0(\kappa r)$ is analytic in $r^2$, the field is smooth at
  the origin. The exterior branches and their derivatives decay
  exponentially by \cite[equations 10.40.2 and 10.40.4]{DLMF}. Continuity at $r=1$ therefore gives $b\in\Hs^1$.
  The transmission conditions cancel the interface distributions, so
  \[
    \frl_{V^0}(b,\phi)=-\lambda_*\inner b\phi
    \qquad\text{for every }\phi\in\Hs^1.
  \]
  This proves $b\in D(A_{V^0})$ and $A_{V^0}b=\lambda_*b$.
\end{proof}

\subsection{Spectral isolation and regularisation}

We next isolate the eigenvalue from Lemma \ref{lem:jump-horizontal-mode}
and regularise the velocity. We compare the original and the regularised operators via their weak formulations, i.e.\ forms, defined as maps $\Hs^{1} \to \Hs^{-1}$. We use the
perturbation theory of sectorial forms and Riesz projections
\cite[Chapters~IV and~VI]{Kato}. Related spectral isolation and contour
arguments for smooth helical dynamos appear in
\cite[Sections~4.2 and~5]{NV25}. We only need to study the spectrum in the open
right half-plane.

Let $w\in\Lp^\infty(\R^2;\R^3)$ have compact support. On
$V=\Hs^1(\R^2;\CC^2)$, with antidual $V^*=\Hs^{-1}(\R^2;\CC^2)$,
we introduce the shifted form operator as
\[
  \bbA_w(z)=F_z+K_w\colon V\to V^*,\qquad
  F_z=z-(\Delta_\perp-4\pi^2).
\]
Here $K_w$ contains the velocity terms of
equation \eqref{eq:jump-sectorial-form}. The Fourier multiplier
$(z+\abs\xi^2+4\pi^2)^{-1}$ shows that $F_z$ is an isomorphism
from $V$ to $V^*$ for $\Re z>0$.

The whole-plane resolvent need not be compact. What is confined to a
bounded region is the velocity perturbation; its compactness between
the form spaces will suffice to isolate the unstable spectrum.

Choose $R_w>0$ such that $\supp w\subset D(0,R_w)$. Rellich's theorem implies that the restriction map
\[
  R\colon\Hs^1(\R^2)\to\Lp^2(D(0,R_w)), \quad Rb=\left. b\right|_{D(0,R_w)}
\]
is compact. Hence its adjoint
\[
  R^*\colon\Lp^2(D(0,R_w))\to\Hs^{-1}(\R^2)
\]
is compact as well.

In the case that the derivative falls on the test function,
we restrict the input to $D(0,R_w)$.
This step is a compact map. Then, multiplication by the bounded coefficient and pairing
with the derivative of the test function defines a bounded map into $\Hs^{-1}(\R^2)$.
Hence the composition is compact. On the other hand, if the derivative falls on the input,
differentiation and multiplication by the coefficient give a bounded map into $\Lp^2(D(0,R_w))$.
Composing this map with $R^*$ gives a compact operator.
Finally, the zeroth-order term can be written as the composition of $R$, multiplication, and $R^*$.
This shows that every term in $K_w$ is compact from $V$ to $V^*$.

Consequently
\[
  \bbA_w(z)=F_z\left(\id+F_z^{-1}K_w\right)
\]
is an analytic Fredholm family in $\{\Re z>0\}$. The sectorial
estimate \eqref{eq:sectorial-estimate-b} makes $\bbA_w(\beta)$ invertible for a sufficiently
large real $\beta$.
The analytic Fredholm theorem \cite[Chapter XI, Corollary 8.4]{GGK90} implies that \(\bbA_w(z)\) is invertible except at a discrete set of points in the open right half-plane. Where it is invertible, its inverse restricted to \(L^2\subset V^*\) equals \((z-A_w)^{-1}\). At each exceptional point, Fredholm index zero gives a non-zero kernel, which the definition of \(A_w\) identifies with \(\ker(z-A_w)\). Hence the spectrum of \(A_w\) in the open right half-plane consists of isolated eigenvalues.

Applying this conclusion to $w=V^0$, we fix a positively oriented
circle $\Gamma$ whose closed interior $\cG$ lies in
$\{\Re z>0\}$ and whose interior contains $\lambda_*$ and no
other spectral point of $A_{V^0}$.

For $0<h<1/2$, choose a smooth radial cut-off $\chi_h$ with
$0\le\chi_h\le1$, equal to one on $[0,1-h]$ and zero on
$[1+h,\infty)$. We define
\begin{equation}
  v_h(y)=\chi_h(\abs y)(-\Omega y_2,\Omega y_1,U).
  \label{eq:smooth-local-velocity}
\end{equation}
Then $v_h\in\C^\infty_c(\R^2;\R^3)$ and
$\Div_\perp(v_h)_\perp=0$, since
$\nabla\chi_h(\abs y)\cdot y^\perp=0$.

\begin{proposition}[Uniform instability of the reference model on $\R^2$]\label{prop:local-seed}
  There exist $h\in(0,1/2)$, an
  interval $I=[\delta_-,\delta_+]$ with $0<\delta_-<1<\delta_+$,
  a function $f_*\in\C^\infty_c(\R^2;\CC^2)$, and $L_*>0$
  with the following properties. For the reference velocity profile on $\R^2$, $v=v_h$, from
  equation \eqref{eq:smooth-local-velocity} and $\delta\in I$, we define the reference operator on $\R^2$
  \begin{equation}
    \cL_\delta b=\delta(\Delta_\perp-4\pi^2)b
    -v_\perp\cdot\nabla_\perp b-2\pi iv_3b
    +(b\cdot\nabla_\perp)v_\perp
    \label{eq:local-generator}
  \end{equation}
  on $\Lp^2(\R^2;\CC^2)$ with domain $\Hs^2(\R^2;\CC^2)$.
  Then $\Gamma\subset\rho(\cL_\delta)$ for every $\delta\in I$,
  and
  \begin{equation}
    \sup_{\substack{\delta\in I\\z\in\Gamma}}
    \norm{(z-\cL_\delta)^{-1}}_{\Lp^2\to\Hs^1}<\infty.
    \label{eq:local-continuation-bound}
  \end{equation}
  The Riesz projections of the reference operators on $\R^2$
  $P_\delta=(2\pi i)^{-1}\int_\Gamma(z-\cL_\delta)^{-1}\dd z$
  satisfy
  \begin{equation}
    \norm{f_*}_2=1,\qquad
    \supp f_*\subset D(0,L_*),\qquad
    \Re\inner{P_\delta f_*}{f_*}>\frac12
    \quad\text{for every }\delta\in I.
    \label{eq:local-spectral-witness}
  \end{equation}
  In particular, each $\cL_\delta$ has spectrum inside $\Gamma$, and
  \begin{equation}
    g_*:=\min_{\zeta\in\cG}\Re\zeta>0.
    \label{eq:spectral-gap}
  \end{equation}
\end{proposition}

\begin{proof}
  We perturb the velocity and diffusivity on the fixed domain $\Hs^1(\R^2;\CC^2)$.
  For $h>0$ and $\delta>0$, we set
  \[
    \frl_{h,\delta}(b,\phi)
    =\frl_{v_h}(b,\phi)
    +(\delta-1)\left(\inner{\nabla_\perp b}{\nabla_\perp\phi}
    +4\pi^2\inner b\phi\right).
  \]
  We write $\bbA_{h,\delta}(z)\colon V\to V^*$ for the
  shifted form operator and $\cL_{h,\delta}$ for the associated
  generator on $\Lp^2$. At $(h,\delta)=(0,1)$, we recover $v_0=V^0$.
  The velocity difference is uniformly bounded and supported
  in an annulus of area $O(h)$, so
  $\norm{v_h-V^0}_4\le Ch^{1/4}$.
  H\"older's inequality and $\Hs^1(\R^2)\hookrightarrow\Lp^4(\R^2)$
  give
  \begin{equation}
    \norm{\bbA_{h,\delta}(z)-\bbA_{0,1}(z)}_{V\to V^*}
    \le C\left(h^{1/4}+\abs{\delta-1}\right).
    \label{eq:form-small}
  \end{equation}
  We estimate each first-order
  term with exponents $4,2,4$ or $4,4,2$, according to which field
  is differentiated. The zeroth-order term is bounded in the same
  way. The diffusion difference is bounded directly on $V\times V$.
  The constant in \eqref{eq:form-small} is independent of
  $z$, $h\in(0,1/2)$ and $\delta$.
  Smallness here measures the width of the transition, not the size of
  the velocity gradients created by smoothing. Those gradients need not
  be small.

  The inverses $\bbA_{0,1}(z)^{-1}\colon V^*\to V$ are
  uniformly bounded on $\Gamma$. Applying a Neumann series, we
  obtain, for $h^{1/4}+\abs{\delta-1}$ sufficiently small,
  \[
    \sup_{z\in\Gamma}
    \norm{\bbA_{h,\delta}(z)^{-1}
      -\bbA_{0,1}(z)^{-1}}_{V^*\to V}
    \le C\left(h^{1/4}+\abs{\delta-1}\right).
  \]
  Restricting to $\Lp^2$ gives the resolvents of $\cL_{h,\delta}$
  on $\Gamma$. Integrating there, we deduce
  \[
    \norm{P_{h,\delta}-P^0}_{2\to2}
    \le C\left(h^{1/4}+\abs{\delta-1}\right),\qquad
    P^0=\frac1{2\pi i}\int_\Gamma(z-A_{V^0})^{-1}\dd z,
  \]
  where $P_{h,\delta}$ denotes the corresponding Riesz projection.

  We choose a unit vector $x\in\operatorname{Ran}P^0$.
  Since $P^0x=x$, density gives a unit function
  $f_*\in\C^\infty_c(\R^2;\CC^2)$ such that
  $\Re\inner{P^0f_*}{f_*}>3/4$. We choose $L_*>0$ with
  $\supp f_*\subset D(0,L_*)$. We now fix $h>0$ and an interval
  $I=[\delta_-,\delta_+]$ about $1$, with $\delta_->0$, so small
  that $\norm{P_{h,\delta}-P^0}_{2\to2}<1/4$ for every
  $\delta\in I$. Writing $v=v_h$, $\cL_\delta=\cL_{h,\delta}$
  and $P_\delta=P_{h,\delta}$ proves
  equation \eqref{eq:local-spectral-witness}.

  It remains to identify the operator domains. An element of the
  form-operator domain belongs to $\Hs^1$. Since $v$ is smooth with
  bounded derivatives, the lower-order terms belong to $\Lp^2$.
  The equation therefore gives $\Delta_\perp b\in\Lp^2$, whence
  $b\in\Hs^2$. The converse follows by
  integration by parts. Moreover, the form inverse estimate gives
  $\norm{(z-\cL_\delta)^{-1}f}_{\Hs^1}\le C\norm f_2$
  uniformly for $(\delta,z)\in I\times\Gamma$. This proves
  equation \eqref{eq:local-continuation-bound}. Finally,
  $P_\delta\ne0$ implies that the spectrum inside $\Gamma$ is
  non-empty. This completes the proof.
\end{proof}

\section{Construction of the velocity field on $\T^3$}
\label{sec:velocity}

\subsection{Choice of the parameters}

In this section, we construct the velocity field on $\T^3$ from the reference velocity profile on $\R^2$
in Proposition~\ref{prop:local-seed}. We keep
\[
  v,\qquad I,\qquad\Gamma,\qquad\cG,
  \qquad g_*,\qquad f_*,\qquad L_*
\]
fixed. For each $n\ge0$, we introduce a horizontal scale $\ell_n \ll 1$,
an integer axial Fourier index $K_n$, and a centre $p_n$. We call the
support of the rescaled reference profile the \emph{core}. Around this
core, we introduce a \emph{cell} $D(p_n,R_n)$ in $\T^2$ of radius
$R_n=\sqrt{\ell_n}$ and add a smoothly cut off axial velocity profile in this disc.
As in Subsection~\ref{sec:ideas}, we call $D(p_n,R_n)$ the $n$-th
cell; its corresponding region in $\T^3$ is $D(p_n,R_n)\times\T$.

We choose a number $q>1$ and a fixed axial cut-off $\vartheta$. The
parameter $\tau$ of the axial shifts will be fixed in Proposition~\ref{prop:background-resolvent}. After that we choose the initial Fourier index
$K_0$ sufficiently large to close the geometric and spectral estimates.

\subsection{Diffusivity intervals and horizontal scales}

The diffusivity intervals are assigned to rescaled reference profiles on $\T^2$ as in
\cite[Section~4.1]{CSV25}.

We choose $1<q<\sqrt{\delta_+/\delta_-}$. For an integer $K_0$
so large that $q+K_0^{-1}<\sqrt{\delta_+/\delta_-}$, we define
recursively
\[
  K_{n+1}=\lceil qK_n\rceil.
\]
Then
\begin{equation}
  q\le\frac{K_{n+1}}{K_n}
  \le q+\frac1{K_0}
  <\sqrt{\frac{\delta_+}{\delta_-}}.
  \label{eq:ratio-choice}
\end{equation}
Every subsequent enlargement of $K_0$ preserves the inequalities.

Set
\[
  \ell_n=\frac1{K_n}.
\]
The diffusivity intervals
\begin{equation}
  \cI_n
  =
  \left[
    \delta_-\frac1{K_n^2},
    \delta_+\frac1{K_n^2}
    \right]
  \label{eq:diffusivity-intervals}
\end{equation}
overlap with their neighbours. Indeed,
\[
  \delta_-\frac1{K_n^2}
  \le
  \delta_+\frac1{K_{n+1}^2}
  \quad\Longleftrightarrow\quad
  \left(\frac{K_{n+1}}{K_n}\right)^2
  \le\frac{\delta_+}{\delta_-},
\]
and the last inequality follows from \eqref{eq:ratio-choice}. Since the
intervals overlap and both endpoints tend to zero,
\[
  \bigcup_{n\ge0}\cI_n=(0,\eps_0],
  \qquad
  \eps_0=\delta_+\frac1{K_0^2}.
\]
Thus every $0<\eps\le\eps_0 $ admits at least one index $n $ for
which $\delta_n:=\eps/\ell_n^2\in I $.

\subsection{Disjoint placement and added axial velocity profiles on $\T^2$}

We place the rescaled reference profiles on $\T^2$ in disjoint discs accumulating at one
point. A related assembly of shrinking velocity fields occurs in
\cite[Section~3.1]{ACM19}. In our work each profile is surrounded by a larger
disc on which we add an axial velocity profile.

Put
\[
  R_n=\sqrt{\ell_n}
  =\sqrt{\frac1{K_n}}.
\]
The core scale is $\ell_n$, whereas $R_n$ is the buffer scale used for
spatial localisation. We define the constant axial values by
\[
  a_n=\frac{\tau}{2\pi K_n}=\frac{\tau}{2\pi}\ell_n.
\]
Then $2\pi K_na_n=\tau$. The parameter $\tau>0$ will be fixed in
Proposition~\ref{prop:background-resolvent}, independently of $K_0$.
The relations
\[
  \frac{\ell_n}{R_n}=\sqrt{\ell_n},
  \qquad
  \frac{\eps}{R_n^2}=\delta_n\ell_n
\]
will make the multiplier and cutoff errors small.
Since
$K_n $ grows at least geometrically,
\[
  R_n\le
  \sqrt{\frac1{K_0}}\,q^{-n/2},
  \qquad
  \sum_{n\ge0}R_n
  \le
  \frac{K_0^{-1/2}}{1-q^{-1/2}},
  \qquad
  \frac{R_n}{\ell_n}
  =\sqrt{K_n}\longrightarrow\infty.
\]

For $K_0 $ sufficiently large,
\begin{equation}
  4\sum_{n\ge0}R_n+R_0<\frac14.
  \label{eq:chart-packing}
\end{equation}
We use the coordinate square
$(-\frac14,\frac14)\times(\frac14,\frac34)\subset\T^2 $;
all discs below lie in this fixed Euclidean chart.
Here $D(p,R) $ denotes a Euclidean disc in the chart, and the
corresponding subset of $\T^3 $ is $D(p,R)\times\T $.
We place the centres at
\[
  p_n=
  \left(4\sum_{m\ge n}R_m,\frac12\right).
\]
The first coordinate of every point in $D(p_n,R_n) $ lies between
\[
  4\sum_{m\ge n}R_m-R_n>0
  \quad\text{and}\quad
  4\sum_{m\ge0}R_m+R_0<\frac14.
\]
Its second coordinate lies in $(1/4,3/4) $ because
$R_n\le R_0<1/4 $. Hence every cell is contained in the stated
Euclidean chart. Moreover,
$\abs{p_n-p_{n+1}}=4R_n>R_n+R_{n+1} $. The centres are ordered on
a line and thus all discs are pairwise
disjoint. Finally, $p_n\to p_\infty=(0,\frac12) $, the unique
accumulation point in the horizontal torus $\T^2 $. The corresponding
accumulation set in $\T^3 $ is the circle $\{p_\infty\}\times\T $.

Fix $\vartheta\in\C^\infty_c([0,\infty))$ such that
\[
  0\le\vartheta\le1,\qquad
  \vartheta=1\ \hbox{on }[0,3/4],\qquad
  \supp\vartheta\subset[0,1).
\]
We make this choice before fixing $\tau$.
Define the rescaled reference profile in the $n$-th cell of $\T^2$ and the added axial velocity profile on $\T^2$ with its smooth cut-off by
\begin{align}
  u_n^{\loc}(y)
   & =
  \ell_n v\left(\frac{y-p_n}{\ell_n}\right),
  \label{eq:local-copy} \\
  u_n^{\axi}(y)
   & =
  a_n\vartheta\left(\frac{\abs{y-p_n}}{R_n}\right)e_3.
  \label{eq:axial-offset}
\end{align}
Because $\vartheta=1 $ on $[0,3/4] $,
$u_n^{\axi}=a_ne_3 $ throughout
$D(p_n,3R_n/4) $. This is the region where the added axial velocity profile
is constant; the axial component of the rescaled reference profile may still vary there.
Moreover,
\begin{align*}
  \norm{\nabla u_n^{\axi}}_\infty
   & \le \frac{a_n}{R_n}\norm{\vartheta'}_\infty, \\
  \eta_{\axi}
   & :=
  \sup_n\frac{a_n}{R_n}\norm{\vartheta'}_\infty
  =
  \frac{\tau}{2\pi\sqrt{K_0}}\norm{\vartheta'}_\infty
  \longrightarrow0,
  \qquad K_0 \to \infty.
\end{align*}
Indeed,
\[
  \frac{a_n}{R_n}
  =\frac{\tau}{2\pi\sqrt{K_n}},
\]
so its supremum is attained at $n=0 $. Thus $\eta_{\axi} $ bounds the
velocity gradient created where the added axial velocity profiles are smoothly cut off.
This bound will be used for regularity of the velocity and its particle
flow.
We choose $R_v\ge L_*$ with
$\supp v\cup\supp\nabla_\perp v\subset D(0,R_v) $.
Since $K_n \ge K_0 $ for all $n$, we have
\begin{equation*}
  \frac{R_v\ell_n}{R_n} = \frac{R_v}{\sqrt{K_n}} \le \frac{R_v}{\sqrt{K_0}}.
\end{equation*}
Thus, for $K_0>(10R_v)^2$ we have
\begin{equation}
  R_v\ell_n<R_n/10
  \qquad\text{for every }n.
  \label{eq:local-support-in-buffer}
\end{equation}
Consequently,
\[
  \supp u_n^{\loc}\subset D(p_n,R_n/10),
  \qquad
  \supp u_n^{\axi}\subset D(p_n,R_n),
  \qquad
  u_n^{\axi}=a_ne_3\quad\text{on }\supp u_n^{\loc}.
\]
The core is the region where the transport
and stretching coefficients of the rescaled reference profile may be non-zero. The region
$D(p_n,3R_n/4)\setminus\supp u_n^{\loc}$ is the
constant-coefficient buffer: the rescaled reference profile has vanished there,
while the added axial velocity profile remains equal to $a_ne_3$.
The cell $D(p_n,R_n)$ contains the core, this
constant-coefficient buffer, and the outer annulus where the added
axial velocity profile is smoothly cut off.
In particular, the combined supports lie in the pairwise disjoint discs
$D(p_n,R_n) $, so at each point at most one summand below is non-zero.
Define the velocity field on $\T^3$ by
\begin{equation}
  u(y,x_3)
  =
  \sum_{n\ge0}
  \left(u_n^{\loc}(y)+u_n^{\axi}(y)\right).
  \label{eq:global-u}
\end{equation}

The placement of the cells and the geometry inside one cell are summarised
in Figure~\ref{fig:multiscale-geometry}.
\begin{figure}[t]
  \centering
  \begin{tikzpicture}[x=1cm,y=1cm,font=\small]
    \draw[thick] (0,0) rectangle (4.4,4.4);
    \node[anchor=north west,font=\bfseries] at (0.18,4.20) {$\T^2$};

    \draw (3.48,2.05) circle (0.54);
    \draw (2.48,2.05) circle (0.36);
    \draw (1.75,2.05) circle (0.24);
    \draw (1.18,2.05) circle (0.15);
    \fill (0.72,2.05) circle (0.035);

    \node (cellzero) at (3.48,0.80) {$D(p_0,R_0) $};
    \draw[->] (cellzero.north) -- (3.48,1.49);

    \node (limitpoint) at (0.72,1.28) {$p_\infty $};
    \draw[->] (limitpoint.north) -- (0.72,1.99);

    \begin{scope}[xshift=8.20cm,yshift=2.20cm]
      \fill[black!7] (0,0) circle (1.55);
      \fill[black!18] (0,0) circle (1.1625);

      \draw[thick] (0,0) circle (1.55);
      \draw[dashed] (0,0) circle (1.1625);
      \draw[thick] (0,0) circle (0.24);

      \node[font=\bfseries] at (0,1.93) {one cell};

      \node[anchor=east,align=right] (core) at (-1.38,0.92)
      {core\\[-1pt]
        \scriptsize scale $\ell_n $};
      \draw[->] (core.east) -- (-0.20,0.07);

      \node[anchor=west,align=left] (axial) at (1.85,0.72)
      {added axial velocity\\[-1pt] profile $u_n^{\axi} $\\[-1pt]
        \scriptsize $=a_ne_3 $ throughout the dashed\\[-1pt]
        \scriptsize disc; smoothly cut off\\[-1pt]
        \scriptsize in the outer annulus};
      \draw[->] (axial.west) -- (0.65,0.58);
      \draw[->] (axial.south west) -- (1.33,-0.30);

      \draw[|-|,thick] (0,-1.75) -- (1.55,-1.75);
      \node[anchor=north] (radius) at (0.775,-2.12)
      {cell radius $R_n $};
    \end{scope}
  \end{tikzpicture}

  \caption{Schematic horizontal geometry. The discs $D(p_n,R_n) $ shrink and accumulate at $p_\infty $; in
    $\T^3 $, the corresponding accumulation set is the circle
    $\{p_\infty\}\times\T $. In each cell, the core has scale
    $\ell_n $. The added axial velocity profile equals $a_ne_3 $ throughout
    $D(p_n,3R_n/4) $, including the core and the surrounding
    constant-coefficient buffer, and is smoothly cut off before the boundary
    of $D(p_n,R_n) $.}
  \label{fig:multiscale-geometry}
\end{figure}

\begin{lemma}[Regularity and incompressibility of the velocity field on $\T^3$]
  \label{lem:global-velocity}
  The velocity field $u $ on $\T^3$ in \eqref{eq:global-u} is a non-zero real-valued autonomous
  vector field belonging to $\W^{1,\infty}(\T^3;\R^3) $. It is
  independent of $x_3 $, and $\Div u=0 $. Moreover, $u $ is smooth
  away from $\{p_\infty\}\times\T $ and is differentiable at every point
  of this circle, with derivative zero.
\end{lemma}

\begin{proof}
  Since $\Div_\perp v_\perp=0 $, every rescaled reference profile on $\T^2$ is
  divergence-free. Each added axial velocity profile on $\T^2$ is independent of $x_3 $, and hence
  is divergence-free as well. Their combined supports are pairwise
  disjoint. Moreover,
  \[
    \norm{\nabla u_n^{\loc}}_\infty\le\norm{\nabla v}_\infty,
    \qquad
    \norm{\nabla u_n^{\axi}}_\infty\le\eta_{\axi}.
  \]
  The vector field is non-zero as it equals $a_ne_3\neq 0 $ on
  \[
    D(p_n,3R_n/4)\setminus\supp u_n^{\loc}.
  \]
  By definition,
  \[
    \norm{u_n^{\loc}}_\infty
    \le\ell_n\norm v_\infty,
    \qquad
    \norm{u_n^{\axi}}_\infty
    \le a_n
    =\frac{\tau}{2\pi}\ell_n.
  \]
  Away from $p_\infty $ the family of supports is locally finite, so the
  sum is smooth there. At the accumulation circle we observe the following.
  If $y\in D(p_n,R_n) $, then
  \[
    \abs{y-p_\infty}
    \ge \abs{p_n-p_\infty}-R_n
    =4\sum_{m\ge n}R_m-R_n
    \ge3R_n.
  \]
  Since $\ell_n=R_n^2 $, the bounds from above give
  \begin{equation*}
    \abs{u(y,x_3)}\le C\ell_n
    =CR_n^2\le C\abs{y-p_\infty}^2.
  \end{equation*}
  Outside the cells the left-hand side vanishes. Thus, in a local
  chart around every $(p_\infty,s) $,
  \[
    \abs{u(y,x_3)-u(p_\infty,s)}
    \le C\abs{(y,x_3)-(p_\infty,s)}^2.
  \]
  Hence the velocity $u $ is differentiable at every point of the accumulation circle,
  and its derivative is zero there. In particular, it is continuous at these points uniformly
  in $x_3 $.

  Thus $u$ is differentiable everywhere, with
  $\norm{Du}\le\norm{\nabla v}_\infty+\eta_{\axi}$.
  Next, we apply the mean-value theorem to its periodic lift to deduce that
  $u$ is globally Lipschitz with this bound. Moreover, its
  divergence vanishes everywhere. Since the weak derivative of a
  Lipschitz function agrees almost everywhere with its classical
  derivative, we conclude that \(u\in\W^{1,\infty}\) and
  $\Div u=0$ in distributions.
\end{proof}

\subsection{Why the construction stops at Lipschitz regularity}
\label{sec:lipschitz-obstruction}

The reference velocity profile $v $ on $\R^2$ is smooth. Regularity of $u$ is lost only because the
velocity field on $\T^3$ has nonvanishing gradients at infinitely many shrinking
horizontal scales. The estimate in the proof of
Lemma~\ref{lem:global-velocity} shows that $u $ is nevertheless differentiable
on the accumulation circle and
\[
  Du(p_\infty,x_3)=0,
  \qquad x_3 \in \T.
\]
Both cutoffs are constant near each cell centre. Hence
\begin{equation*}
  Du(p_n,x_3)=
  \begin{pmatrix}
    0      & -\Omega & 0 \\
    \Omega & 0       & 0 \\
    0      & 0       & 0
  \end{pmatrix}\neq0.
\end{equation*}
Since $p_n\to p_\infty$, the derivative $Du$ is discontinuous at every
point of $\{p_\infty\}\times\T$.

A continuous representative of the weak gradient would agree everywhere
with $Du$. Hence no such representative exists.

\section{Particle dynamics and absence of exponential stretching}
\label{sec:particle}

Let $\Phi_t\colon\T^3\to\T^3 $ denote the unique particle flow generated by
the velocity field $u $ on $\T^3$:
\[
  \frac{\dd}{\dd t}\Phi_t(x)=u(\Phi_t(x)),
  \qquad
  \Phi_0(x)=x.
\]
The existence and uniqueness of this flow for every $t\in\R $ follow from the global
Lipschitz regularity provided in Lemma~\ref{lem:global-velocity}. The special form
of the velocity in each cell allows us to say even more.

\begin{proposition}[Explicitly solvable particle dynamics and zero entropy]
  \label{prop:particle-flow}
  Let $u $ be the velocity field on $\T^3$ in \eqref{eq:global-u}. There is a
  constant $C_{\LL}\ge1 $ such that
  \begin{equation}
    \Lip(\Phi_t)
    +\Lip(\Phi_t^{-1})
    \le C_{\LL}(1+\abs t),
    \qquad t \in \R.
    \label{eq:linear-flow-lipschitz}
  \end{equation}
  In particular,
  \[
    \lim_{\abs t\to\infty}
    \frac1{\abs t}\log\Lip(\Phi_t)=0,
  \]
  and every time map has zero topological entropy:
  \[
    h_{\mathrm{top}}(\Phi_t)=0,
    \qquad t \in \R.
  \]

  Using the reference velocity profile on $\R^2$ in equation \eqref{eq:smooth-local-velocity}, we
  write the particle flow explicitly. Let $\mathsf R_\alpha$ denote
  rotation through angle
  $\alpha $ in the horizontal plane. For the $n $-th cell set
  \begin{equation}
    \omega_n(r)=\Omega\chi_h\left(\frac r{\ell_n}\right),
    \qquad
    g_n(r)=\ell_nU\chi_h\left(\frac r{\ell_n}\right)
    +a_n\vartheta\left(\frac r{R_n}\right).
    \label{eq:particle-cell-profiles}
  \end{equation}
  Here $\omega_n(r)$ and $g_n(r)$ are, respectively, the angular and axial velocities at radius $r$ in the $n$-th cell.
  If $y\in D(p_n,R_n) $, $r=\abs{y-p_n} $, and $x_3\in\T $, then
  \begin{equation}
    \Phi_t(y,x_3)
    =
    \left(
    p_n+\mathsf R_{t\omega_n(r)}(y-p_n),
    \ x_3+t g_n(r)
    \right),
    \label{eq:explicit-particle-flow}
  \end{equation}
  where the second component is understood modulo $1 $. At $y=p_n $,
  the horizontal component is $p_n $, while the axial component
  translates at speed $g_n(0)=\ell_nU+a_n $. Every point outside the
  cells, on their boundaries, and of the
  accumulation circle $\{p_\infty\}\times\T $ is fixed.
\end{proposition}

\begin{proof}
  Substitution of \eqref{eq:smooth-local-velocity} into
  \eqref{eq:local-copy} shows that in the $n $-th cell
  \[
    u(y,x_3)
    =\left(\omega_n(r)(y-p_n)^\perp,g_n(r)\right),
    \qquad r=\abs{y-p_n}.
  \]
  Along a particle trajectory,
  \[
    \frac{\dd}{\dd t}\abs{y-p_n}^2
    =2\omega_n(r)(y-p_n)\cdot(y-p_n)^\perp=0.
  \]
  Thus $r $ is invariant, and the polar angle in the horizontal plane
  and the axial coordinate evolve at the constant angular and axial velocities
  $\omega_n(r) $ and $g_n(r) $,
  respectively. This proves \eqref{eq:explicit-particle-flow}. It also
  shows that no trajectory crosses a cell boundary. The profiles vanish
  in a neighbourhood of $\partial D(p_n,R_n) $, so the formula agrees with the identity there.

  The velocity vanishes outside of the cells. Hence, the flow is the identity there.
  In particular, every point of the accumulation circle $\{p_\infty\}\times\T$ is fixed by the flow.

  We next prove that
  \[
    \norm{D\Phi_t}\le C(1+\abs t)
  \]
  in each cell, with a constant $C$ independent of the cell index $n$. For this purpose, we differentiate equation \eqref{eq:explicit-particle-flow}.

  Put $z=y-p_n $, and for $r>0 $ write $e_r=z/r $. Let
  $\mathsf Jz=z^\perp $ denote rotation by $90^\circ $. Then
  \[
    D_y\!\left(\mathsf R_{t\omega_n(r)}z\right)
    =\mathsf R_{t\omega_n(r)}
    \left[\id_{\R^2}+t\omega_n'(r)\mathsf Jz\otimes e_r\right].
  \]
  The profiles are constant near $r=0 $, so this formula extends smoothly. Moreover,
  \begin{align}
    \sup_{n\ge0}\sup_{r\ge0}r\abs{\omega_n'(r)}
     & \le
    \abs\Omega\sup_{s\ge0}s\abs{\chi_h'(s)},
    \label{eq:omega-shear-bound} \\
    \sup_{n\ge0}\sup_{r\ge0}\abs{g_n'(r)}
     & \le
    \abs U\norm{\chi_h'}_\infty+\eta_{\axi}.
    \label{eq:axial-shear-bound}
  \end{align}
  Hence
  \[
    \norm{D\Phi_t}\le C(1+\abs t)
  \]
  with a constant independent of $n $.

  By Lemma \ref{lem:global-velocity}, the natural lift of $\Phi_t$
  is globally Lipschitz for every fixed $t$. Its derivative is obtained by differentiating \eqref{eq:explicit-particle-flow} and equals the identity almost
  everywhere outside their union. Hence the preceding uniform bound
  gives
  \[
    \Lip(\Phi_t)\le C(1+\abs t).
  \]
  Applying the same estimate at $-t$ and using
  $\Phi_t^{-1}=\Phi_{-t}$, we obtain
  equation \eqref{eq:linear-flow-lipschitz}, after increasing the
  constant.

  It remains to prove that every time map has zero topological entropy.
  Fix $s\in\R$ and an integer $m\ge1$. Applying the entropy formula
  for iterates and the Lipschitz entropy bound from
  \cite[Section~2.5f(3) and Proposition~2.5.5]{HasselblattKatok02}, we obtain
  \[
    m h_{\mathrm{top}}(\Phi_s)
    =h_{\mathrm{top}}(\Phi_{ms})
    \le3\log\left(C_{\LL}(1+m\abs s)\right),
  \]
  where we have used equation \eqref{eq:linear-flow-lipschitz}.
  Dividing by $m$ and letting $m\to\infty$, we conclude that
  $h_{\mathrm{top}}(\Phi_s)=0$. This completes the proof.
\end{proof}

\begin{corollary}[No exponential growth for ideal induction]
  \label{cor:ideal-growth}
  For the velocity field $u $ on $\T^3$ in \eqref{eq:global-u}, let
  $\cT_0(t) $ denote the solution group on
  $\Lp^2_\sigma(\T^3) $ for the ideal induction equation
  \[
    \partial_tB+(u\cdot\nabla)B-(B\cdot\nabla)u=0,
    \qquad
    \Div B=0.
  \]
  Then
  \begin{equation}
    \frac{1}{C_{\LL}(1+\abs t)}
    \le
    \norm{\cT_0(t)}_{\Lp^2_\sigma\to\Lp^2_\sigma}
    \le
    C_{\LL}(1+\abs t),
    \qquad t \in \R,
    \label{eq:ideal-group-growth}
  \end{equation}
  where $C_{\LL} $ is the constant in
  \eqref{eq:linear-flow-lipschitz}. Consequently the ideal induction
  equation has zero exponential growth rate in operator norm:
  \begin{equation}
    \gamma_{\rm id}(u)
    :=
    \lim_{t\to\infty}
    \frac1t
    \log
    \norm{\cT_0(t)}_{\Lp^2_\sigma\to\Lp^2_\sigma}
    =0.
    \label{eq:zero-ideal-growth}
  \end{equation}
  Moreover, for every $B_{\inn}\in\Lp^2_\sigma(\T^3)$ and $t\in\R$,
  \[
    \frac{\norm{B_{\inn}}_2}{C_{\LL}(1+\abs t)}
    \le\norm{\cT_0(t)B_{\inn}}_2
    \le C_{\LL}(1+\abs t)\norm{B_{\inn}}_2.
  \]
  In particular, every non-zero ideal solution has exponential growth
  rate zero.
\end{corollary}

\begin{proof}
  We derive the Cauchy formula directly from the Lipschitz flow. We work with the
  the periodic lift of \(u\) and the corresponding equivariant lift of \(\Phi_t\). The Cauchy--Lipschitz theorem,
  Rademacher's theorem, and the chain rule give
  \begin{equation}
    D\Phi_t(a)
    =\id+\int_0^t
    (\nabla u)(\Phi_s(a))D\Phi_s(a)\dd s
    \label{eq:flow-variational-equation}
  \end{equation}
  for almost every $a $. These facts are recalled in
  \cite[Section~2]{AmbrosioCrippa14}; see also the classical Cauchy formula in
  \cite[Chapter~V, Section~1.B]{AK98}. Liouville's formula and
  $\Div u=0 $ give $\det D\Phi_t=1 $ almost everywhere. Since $\Phi_t $
  is bi-Lipschitz, the area formula shows that it preserves Lebesgue measure.

  For $B_{\inn}\in\Lp^2_\sigma(\T^3) $, define
  \begin{equation}     \label{eq:ideal-cauchy-formula}
    \left(\cT_0(t)B_{\inn}\right)(\Phi_t(a))
    =
    D\Phi_t(a)B_{\inn}(a)
  \end{equation}
  for almost every $a $. This defines an $\Lp^2 $ field, and
  preservation of measure gives
  \[
    \norm{\cT_0(t)B_{\inn}}_2
    \le
    \Lip(\Phi_t)\norm{B_{\inn}}_2.
  \]

  The formula preserves the divergence constraint. Indeed, if
  $\psi\in\C^\infty(\T^3) $, then
  \begin{align*}
    \int_{\T^3}\cT_0(t)B_{\inn}\cdot\nabla\psi\dd x
     & =\int_{\T^3}D\Phi_t(a)B_{\inn}(a)
    \cdot\nabla\psi(\Phi_t(a))\dd a      \\
     & =\int_{\T^3}B_{\inn}(a)
    \cdot\nabla(\psi\circ\Phi_t)(a)\dd a=0.
  \end{align*}
  Here $\psi\circ\Phi_t\in\Hs^1 $, and the distributional divergence
  identity extends to such test functions by density.

  Formula \eqref{eq:ideal-cauchy-formula} solves the ideal induction
  equation. Let $\varphi\in\C^\infty(\T^3;\CC^3) $. By measure
  preservation,
  \[
    \inner{\cT_0(t)B_{\inn}}{\varphi}
    =\int_{\T^3}D\Phi_t(a)B_{\inn}(a)
    \cdot\overline{\varphi(\Phi_t(a))}\dd a.
  \]
  Differentiating and using \eqref{eq:flow-variational-equation} gives
  \begin{align*}
    \frac{\dd}{\dd t}
    \inner{\cT_0(t)B_{\inn}}{\varphi}
    =\inner{(\cT_0(t)B_{\inn}\cdot\nabla)u}{\varphi}
    +\inner{\cT_0(t)B_{\inn}}
     {(u\cdot\nabla)\varphi}
  \end{align*}
  for almost every $t $, which is the weak form of the equation.

  The flow law and the chain rule give
  \[
    D\Phi_{t+s}(a)
    =D\Phi_t(\Phi_s(a))D\Phi_s(a)
  \]
  almost everywhere. Hence
  $\cT_0(t+s)=\cT_0(t)\cT_0(s) $ and
  $\cT_0(-t)=\cT_0(t)^{-1} $.

  Uniqueness follows from the same characteristic argument. If
  $W\in\C([-T,T];\Lp^2) $ is a distributional solution with $W(0)=0 $,
  then \cite[Proposition~2.3]{AmbrosioCrippa14}, applied componentwise after
  periodic extension, gives
  \[
    W(t,\Phi_t(a))
    =\int_0^t
    (\nabla u)(\Phi_s(a))W(s,\Phi_s(a))\dd s
  \]
  for almost every $a $. Gronwall's inequality, forward and backward in
  time, gives $W=0 $.

  It remains to check strong continuity. Put
  $L=\norm{\nabla u}_\infty $. The flow equation and
  \eqref{eq:flow-variational-equation} imply
  \[
    \norm{\Phi_t-\id}_\infty\le\norm u_\infty\abs t,
    \qquad
    \norm{D\Phi_t-\id}_{\Lp^\infty}\le\e^{L\abs t}-1.
  \]
  For smooth $B $, preservation of measure and
  \eqref{eq:ideal-cauchy-formula} therefore give
  \[
    \norm{\cT_0(t)B-B}_2^2
    =\int_{\T^3}
    \abs{D\Phi_t(a)B(a)-B(\Phi_t(a))}^2\dd a
    \longrightarrow0.
  \]
  Density and the operator bound near zero extend the convergence to every
  $B\in\Lp^2_\sigma $. Thus $\cT_0(t) $ is the solution group on
  $\Lp^2_\sigma $.

  Finally, equation \eqref{eq:ideal-cauchy-formula} and
  Proposition \ref{prop:particle-flow} give
  \[
    \norm{\cT_0(t)B_{\inn}}_2
    \le C_{\LL}(1+\abs t)\norm{B_{\inn}}_2.
  \]
  Since $\cT_0(-t)\cT_0(t)=\id$, we also have
  \[
    \norm{B_{\inn}}_2
    =\norm{\cT_0(-t)\cT_0(t)B_{\inn}}_2
    \le C_{\LL}(1+\abs t)
    \norm{\cT_0(t)B_{\inn}}_2.
  \]
  This gives \eqref{eq:ideal-group-growth}. Taking logarithms and
  dividing by $t$ proves both zero-growth conclusions as
  $t\to\infty$. This completes the proof.
\end{proof}

\section{The horizontal operator on $\T^2$ and the scaling identity}
\label{sec:axial}

In this section, we reduce the spectral problem to the two horizontal
components. For each diffusivity we choose an axial Fourier index and
identify the operator on $\R^2$ for the selected cell with a shifted version of the reference operator on $\R^2$ from
Section \ref{sec:local-model}. The axial magnetic component will be
recovered in Section \ref{sec:finish}.

In the spectral analysis below, we distinguish the reference model
on $\R^2$, the model on $\R^2$ for the selected cell,
and the horizontal problem on $\T^2$
associated with the velocity profile $u$ on \(\T^2\). We use the same
role and spatial-domain terminology for the corresponding operators, resolvents,
approximate inverses and Riesz projections.

Let $\eps>0$ and $K\in\Z$. Since $u$ is independent of $x_3$, the
horizontal equations in the Fourier mode $\e^{2\pi iKx_3}$ are closed.
We define the horizontal operator on $\T^2$, \(A_{\eps,K}\), on $\Lp^2(\T^2;\CC^2)$ by
\begin{align}\label{eq:torus-sector}
  A_{\eps,K}h
              ={} & \eps\left(\Delta_\perp-(2\pi K)^2\right)h
  -u_\perp\cdot\nabla_\perp h-2\pi iKu_3h
  +(h\cdot\nabla_\perp)u_\perp,
  \notag                                                      \\
              & D(A_{\eps,K})=\Hs^2(\T^2;\CC^2).
\end{align}
No constraint is imposed on $h$. Multiplication by a
smooth cut-off preserves the domain. We note that the axial velocity
enters only through the imaginary multiplication operator $-2\pi iKu_3$.
Its derivatives do not play a role in $A_{\eps,K}$.
For $K\neq 0$, a non-zero horizontal divergence is not an obstruction:
it can be balanced by the axial component of the magnetic field; see Lemma \ref{lem:axial-recovery}. This is
why the horizontal localisation below requires no correction to enforce
a divergence constraint.

\begin{lemma}[The horizontal operator on $\T^2$ associated with $w$]
  \label{lem:torus-sector-basic}
  Let $\eps>0$, $K\in\Z$ and
  $w=(w_\perp,w_3)\in\W^{1,\infty}(\T^2;\R^3)$ satisfy
  $\Div_\perp w_\perp=0$. We denote by \(A_{\eps,K}[w]\) the
  horizontal operator on $\T^2$ associated with $w$, given by equation \eqref{eq:torus-sector} with $u$ replaced by
  $w$. Then $A_{\eps,K}[w]$ is closed and its resolvent
  set is nonempty. Its graph norm is equivalent to the $\Hs^2$-norm, and its
  resolvent is compact. Moreover, the map
  \[
    z-A_{\eps,K}[w]\colon\Hs^2(\T^2;\CC^2)\longrightarrow
    \Lp^2(\T^2;\CC^2)
  \]
  is Fredholm of index zero for every $z\in\CC$.
\end{lemma}

\begin{proof}
  We first note that the difference between $z-A_{\eps,K}[w]$ and
  $\id-\eps\Delta_\perp$ is bounded from $\Hs^1$ to $\Lp^2$.
  Since $\Hs^2(\T^2)$ embeds compactly into $\Hs^1(\T^2)$,
  this difference is compact from $\Hs^2$ to $\Lp^2$.
  The operator $\id-\eps\Delta_\perp$ is an isomorphism between
  these spaces. Hence $z-A_{\eps,K}[w]$ is Fredholm of index zero.

  Since $\Div_\perp w_\perp=0$ and $w_3$ is real-valued, we have
  \[
    \Re\inner{(z-A_{\eps,K}[w])h}h
    \ge\eps\norm{\nabla_\perp h}_2^2
    +\left(\Re z+\eps(2\pi K)^2
    -\norm{\nabla_\perp w_\perp}_\infty\right)\norm h_2^2.
  \]
  Thus $z_0-A_{\eps,K}[w]$ is injective for any real
  $z_0>\norm{\nabla_\perp w_\perp}_\infty$.
  The Fredholm index gives surjectivity. The inverse mapping theorem gives
  $\norm h_{\Hs^2}\le C\norm{(z_0-A_{\eps,K}[w])h}_2$.
  Since $A_{\eps,K}[w]\colon\Hs^2\to\Lp^2$ is bounded, its graph
  norm is equivalent to the $\Hs^2$-norm. Hence the operator is closed.
  The compact embedding $\Hs^2\hookrightarrow\Lp^2$ gives compactness
  of its resolvent. This proves the claim.
\end{proof}

Let us now fix \(0<\eps\le\eps_0\). By \eqref{eq:diffusivity-intervals}, there is an index $n=n(\eps)$ such that
\[
  \delta_n=\frac{\eps}{\ell_n^2}\in I.
\]
We call the corresponding cell the selected cell and work at the selected axial Fourier index $K_n$.

The model on $\R^2$ for the selected cell retains the rescaled reference profile $u_n^{\loc}$ and extends the
constant axial velocity $a_ne_3$ to all of $\R^2$. On
$\Lp^2(\R^2;\CC^2)$, we define
\begin{align}
  A^{(n)}_{\eps,n}h
                    ={} & \eps\left(\Delta_\perp-(2\pi K_n)^2\right)h
  -(u_n^{\loc})_\perp\cdot\nabla_\perp h\notag                        \\
                    & -2\pi iK_n\left((u_n^{\loc})_3+a_n\right)h
  +(h\cdot\nabla_\perp)(u_n^{\loc})_\perp,
  \label{eq:selected-whole-plane-operator}
\end{align}
with domain $\Hs^2(\R^2;\CC^2)$. We also define the unitary map
\[
  U_n\colon\Lp^2(\R^2_Y;\CC^2)\longrightarrow\Lp^2(\R^2_y;\CC^2),
  \qquad
  (U_nf)(y)=\ell_n^{-1}f\left(\frac{y-p_n}{\ell_n}\right).
\]
Its inverse is given by $(U_n^{-1}g)(Y)=\ell_ng(p_n+\ell_nY)$.

\begin{lemma}[Scaling]
  \label{lem:exact-renormalization}
  Let $n=n(\eps)$ be chosen as above. Then
  \begin{equation}
    U_n^{-1}A^{(n)}_{\eps,n}U_n=\cL_{\delta_n}-i\tau
    \label{eq:scaling}
  \end{equation}
  as operators with domain $\Hs^2(\R^2;\CC^2)$, where the right-hand side is the shifted reference operator on $\R^2$.
\end{lemma}

\begin{proof}
  We use $u_n^{\loc}(p_n+\ell_nY)=\ell_nv(Y)$,
  $K_n\ell_n=1$ and $2\pi K_na_n=\tau$. For
  $f\in\Hs^2(\R^2;\CC^2)$, the change of variables
  $y=p_n+\ell_nY$ gives
  \begin{align*}
    U_n^{-1}A^{(n)}_{\eps,n}U_nf
                              ={} & \delta_n(\Delta_Y-4\pi^2)f-v_\perp\cdot\nabla_Yf
    -2\pi iv_3f                                                                      \\
                                 & +(f\cdot\nabla_Y)v_\perp-i\tau f
    =(\cL_{\delta_n}-i\tau)f.
  \end{align*}
  This proves the claim.
\end{proof}

Consequently, the shifted contour and its closed interior are
\[
  \Gamma_\tau=\{z-i\tau\, : \,z\in\Gamma\},
  \qquad
  \cG_\tau=\{z-i\tau\, : \,z\in\cG\}.
\]
Both sets are independent of $n$ and $\eps$, and their real parts are
unchanged.

\section{A uniform estimate for the background resolvent on $\T^2$}
\label{sec:localization}

In this section, we remove the rescaled reference profile in the selected cell of $\T^2$ and estimate the
background resolvent on $\T^2$. We retain all other rescaled reference profiles on $\T^2$ and every
added axial velocity profile on $\T^2$. Testing with a rational function of the axial
velocity gives one coercive estimate for the background problem on $\T^2$.
This is inspired by the complex-potential
multipliers in \cite[Section~2.3.1]{AH15} and the weighted coercivity
estimate in \cite[Theorem~3.3]{KRRS17}.

Let $\eps>0$ and $n\ge0$ satisfy
$\delta_n=\eps/\ell_n^2\in I$. We define the background velocity profile on $\T^2$ by
\begin{equation}
  u_n^{\bg}=u-u_n^{\loc}
  =\sum_{m\ne n}u_m^{\loc}+\sum_{m\ge0}u_m^{\axi}.
  \label{eq:background-velocity}
\end{equation}
We consider the background operator on $\T^2$
\begin{equation}
  \begin{gathered}
    A^{\bg}_{\eps,n}h
    =\eps\left(\Delta_\perp-(2\pi K_n)^2\right)h-(u_n^{\bg})_\perp\cdot\nabla_\perp h
    -2\pi iK_n(u_n^{\bg})_3 h+(h\cdot\nabla_\perp)(u_n^{\bg})_\perp, \\
    D(A^{\bg}_{\eps,n})=\Hs^2(\T^2;\CC^2),
  \end{gathered}
  \label{eq:background-operator}
\end{equation}
on $\Lp^2(\T^2;\CC^2)$. The background operator on $\T^2$, $A^{\bg}_{\eps,n}$, and
the horizontal operator on $\T^2$, $A_{\eps,K_n}$, agree outside $\supp u_n^{\loc}$.

\begin{proposition}[Uniform background resolvent on $\T^2$]
  \label{prop:background-resolvent}
  There exists $\tau_0>0$, depending only on $v,I,\cG,g_*$ and $q$, with
  the following property. Fix $\tau\ge\tau_0$. For all sufficiently large
  $K_0$, every $n\ge0$ and every $\eps=\delta\ell_n^2$, $\delta\in I$,
  satisfy
  \[
    \cG_\tau\subset\rho(A^{\bg}_{\eps,n}).
  \]
  Moreover, for $z\in\cG_\tau$ and $f\in\Lp^2(\T^2;\CC^2)$,
  \begin{equation}
    \norm{(z-A^{\bg}_{\eps,n})^{-1}f}_2
    +\sqrt\eps\,
    \norm{\nabla_\perp(z-A^{\bg}_{\eps,n})^{-1}f}_2
    \le C\norm f_2.
    \label{eq:bg-resolvent}
  \end{equation}
  The constant $C$ is independent of $n$, $\delta$, $z$ and the final
  enlargement of $K_0$.
\end{proposition}

\begin{proof}
  We define
  \[
    r_m=\frac{K_n}{K_m}=\frac{\ell_m}{\ell_n},
    \qquad c_q=\frac{q-1}{q+1},
    \qquad M=\max_{\zeta\in\cG}\abs{\Im\zeta},
    \qquad d(K_0)=\sup_{m\ge0}\frac{\ell_m}{R_m}.
  \]
  Then $r_m\ge q$ for $m<n$, $r_m\le q^{-1}$ for $m>n$, and
  $d(K_0)=\sqrt{\ell_0}\to0$ as $K_0\to\infty$. We set
  \begin{equation}
    s_n(y)=\sum_{m\ge0}r_m
    \vartheta\left(\frac{\abs{y-p_m}}{R_m}\right),
    \qquad
    F(s)=\frac{s-1}{(1+s)^2},
    \qquad \beta_n=F(s_n).
    \label{eq:background-multiplier}
  \end{equation}
  The summands in $s_n$ have disjoint supports. Their partial sums
  converge uniformly and have gradients bounded by
  $C d(K_0)/\ell_n$. Hence, we have $s_n\in\W^{1,\infty}(\T^2)$ for each
  fixed $n$. The same holds for $\beta_n$.

  Let $h\in\Hs^2(\T^2;\CC^2)$ and $z=\zeta-i\tau$, with
  $\zeta\in\cG$. We pair $(z-A^{\bg}_{\eps,n})h$ with
  $(1+i\beta_n)h\in\Hs^1$ and perform an integratio by parts. Since
  $\Div_\perp(u_n^{\bg})_\perp=0$, the
  transport term that contains no factor $\beta_n$ has zero real part. Thus, we have
  \begin{align}
     & \Re\inner{(z-A^{\bg}_{\eps,n})h}{(1+i\beta_n)h}\notag \\
     & \quad=
    \eps\Re\inner{\nabla_\perp h}
           {\nabla_\perp((1+i\beta_n)h)}\notag               \\
     & \qquad+
    \int_{\T^2}
    \left[\Re\zeta+\eps(2\pi K_n)^2
      +\beta_n\left(\Im\zeta+\tau(s_n-1)\right)\right]
    \abs h^2\dd y\notag                                      \\
     & \qquad+
    \Re\inner{(u_n^{\bg})_\perp\cdot\nabla_\perp h}
       {i\beta_n h}\notag           \\
     & \qquad-
    \Re\inner{(h\cdot\nabla_\perp)(u_n^{\bg})_\perp}
       {(1+i\beta_n)h}\notag
    +
    2\pi K_n\int_{\T^2}\beta_n
    \sum_{m\ne n}(u_m^{\loc})_3\abs h^2\dd y\notag           \\
     & \quad=:I_1+I_2+I_3+I_4.
    \label{eq:background-testing}
  \end{align}

  We first estimate $I_1$. Since $s_n\ge0$ and $F'$ is bounded on
  $[0,\infty)$, the chain rule gives
  \begin{equation}
    \abs{\beta_n}\le1,
    \qquad
    \sqrt\eps\,\norm{\nabla_\perp\beta_n}_\infty
    \le C\sqrt{\delta_+}\,d(K_0).
    \label{eq:background-multiplier-bounds}
  \end{equation}
  Moreover,
  \[
    \Re\inner{\nabla_\perp h}
    {\nabla_\perp((1+i\beta_n)h)}
    \ge\norm{\nabla_\perp h}_2^2
    -\int_{\T^2}\abs{\nabla_\perp\beta_n}
    \abs h\abs{\nabla_\perp h}\dd y.
  \]
  Applying \eqref{eq:background-multiplier-bounds}, followed
  by Young's inequality, we deduce
  \[
    \eps\int\abs{\nabla_\perp\beta_n}\abs h
    \abs{\nabla_\perp h}\dd y
    \le\frac\eps4\norm{\nabla_\perp h}_2^2
    +C\delta_+d(K_0)^2\norm h_2^2.
  \]
  Hence
  \[
    I_1\ge\frac{3\eps}{4}\norm{\nabla_\perp h}_2^2
    -C\delta_+d(K_0)^2\norm h_2^2.
  \]

  We next estimate the axial phase in $I_2$. Write
  $\cC_n=\bigcup_{m\ne n}\supp u_m^{\loc}$.
  We have $s_n=r_m$ on $\supp u_m^{\loc}$. Thus
  \[
    \left|\frac{s_n-1}{s_n+1}\right|\ge c_q
    \qquad\text{on }\cC_n.
  \]
  For $\zeta\in\cG$, Young's inequality yields
  \begin{align}
    \beta_n\left(\Im\zeta+\tau(s_n-1)\right)
     & \ge\frac\tau2
    \left(\frac{s_n-1}{s_n+1}\right)^2
    -\frac{M^2}{2\tau(1+s_n)^2}\notag      \\
     & \ge\frac{\tau c_q^2}{2}\one_{\cC_n}
    -\frac{M^2}{2\tau}.
    \label{eq:background-phase-positive}
  \end{align}
  This estimate holds throughout $\T^2$, including every outer cut-off
  annulus of the added axial velocity profiles. Since $\Re\zeta\ge g_*$ by \eqref{eq:spectral-gap}, we obtain
  \[
    I_2\ge
    \left(g_*-\frac{M^2}{2\tau}\right)\norm h_2^2
    +\frac{\tau c_q^2}{2}\int_{\cC_n}\abs h^2\dd y.
  \]

  To estimate $I_3$, we collect some pointwise bounds.
  Using $|rF(r)|\le1$ for
  $r\ge0$, the amplitude bound $|u_m^{\loc}|\le C\ell_m$,
  and $2\pi K_n\ell_m=2\pi r_m$, we obtain
  \begin{align}
    \abs{\beta_n(u_n^{\bg})_\perp}
     & \le C\ell_n\one_{\cC_n}
    \le C\sqrt\eps\,\one_{\cC_n},\notag \\
    \abs{\nabla_\perp(u_n^{\bg})_\perp}
    +\left|2\pi K_n\beta_n\sum_{m\ne n}(u_m^{\loc})_3\right|
     & \le C\one_{\cC_n}.
    \label{eq:background-local-coefficients}
  \end{align}
  Here $\ell_n\le\delta_-^{-1/2}\sqrt\eps$, and the constants depend
  only on $v$ and $\delta_-$.
  Applying \eqref{eq:background-local-coefficients}, followed
  by Young's inequality, we deduce
  \begin{align*}
    \abs{I_3}
     & \le\int\abs{\beta_n(u_n^{\bg})_\perp}
    \abs h\abs{\nabla_\perp h}\dd y           \\
     & \le\frac\eps4\norm{\nabla_\perp h}_2^2
    +C\int_{\cC_n}\abs h^2\dd y.
  \end{align*}

  Finally, the stretching term and the axial multiplier from the rescaled reference profiles on $\T^2$ in $I_4$
  contribute at most $C\int_{\cC_n}|h|^2\dd y$ in absolute value, by \eqref{eq:background-multiplier-bounds} and
  \eqref{eq:background-local-coefficients}.

  Combining the estimates for $I_1,\ldots,I_4$, we obtain
  \begin{align}
    \Re\inner{(z-A^{\bg}_{\eps,n})h}{(1+i\beta_n)h}
    \ge{} & \frac\eps2\norm{\nabla_\perp h}_2^2\notag \\
          & +\left(g_*-\frac{M^2}{2\tau}
    -C_1\delta_+d(K_0)^2\right)\norm h_2^2\notag      \\
          & +\left(\frac{\tau c_q^2}{2}-C_0\right)
    \int_{\cC_n}\abs h^2\dd y.
    \label{eq:background-coercivity}
  \end{align}
  The constants $C_0,C_1$ are independent of $\tau$, $K_0$, $n$ and
  $\delta_n\in I$.

  We choose $\tau$ so that
  \begin{equation}
    \frac{M^2}{2\tau}\le\frac{g_*}{4},
    \qquad
    \frac{\tau c_q^2}{2}\ge C_0+1.
    \label{eq:axial-separation-choice}
  \end{equation}
  We then enlarge $K_0$ until
  \begin{equation}
    C_1\delta_+d(K_0)^2
    =C_1\delta_+\ell_0\le g_*/4.
    \label{eq:background-diffusion-small}
  \end{equation}
  Consequently, for every $h\in\Hs^2(\T^2;\CC^2)$ we have
  \[
    \frac\eps2\norm{\nabla_\perp h}_2^2
    +\frac{g_*}2\norm h_2^2
    \le\Re\inner{(z-A^{\bg}_{\eps,n})h}
    {(1+i\beta_n)h}.
  \]
  In particular, $z-A^{\bg}_{\eps,n}$ is injective.
  The background velocity profile $u_n^{\bg}$ on $\T^2$ is real-valued, belongs to
  $\W^{1,\infty}$ and satisfies
  $\Div_\perp(u_n^{\bg})_\perp=0$.
  Applying Lemma \ref{lem:torus-sector-basic}, we deduce that
  $z-A^{\bg}_{\eps,n}\colon\Hs^2\to\Lp^2$ has index zero and
  is therefore surjective. Thus $z\in\rho(A^{\bg}_{\eps,n})$.
  For $h=(z-A^{\bg}_{\eps,n})^{-1}f$, the preceding estimate
  and $\abs{\beta_n}\le1$ give
  \[
    \frac\eps2\norm{\nabla_\perp h}_2^2
    +\frac{g_*}2\norm h_2^2
    \le\sqrt2\norm f_2\norm h_2.
  \]
  This proves \eqref{eq:bg-resolvent}.
\end{proof}

Let us also state the corresponding estimate for the operator on $\R^2$
for the selected cell.

\begin{corollary}[Resolvent on $\R^2$ for the selected cell]
  \label{cor:scaled-local-resolvents}
  Let $\eps=\delta_n\ell_n^2$ with $\delta_n\in I$. For
  $z\in\Gamma_\tau$, the resolvent on $\R^2$ for the selected cell
  $\cR_n(z)=(z-{A}^{(n)}_{\eps,n})^{-1}$ exists and satisfies
  \begin{equation}     \label{eq:local-gradient}
    \norm{\cR_n(z)f}_2
    +\sqrt\eps\,\norm{\nabla_\perp\cR_n(z)f}_2
    \le C\norm f_2,
    \qquad f\in\Lp^2(\R^2;\CC^2).
  \end{equation}
  The constant is independent of $n$, $\delta_n$ and $z$.
\end{corollary}

\begin{proof}
  By Lemma \ref{lem:exact-renormalization},
  \[
    \cR_n(z)=U_n(z+i\tau-\cL_{\delta_n})^{-1}U_n^{-1}.
  \]
  The map $U_n$ is unitary and
  $\norm{\nabla_\perp U_nf}_2=\ell_n^{-1}\norm{\nabla f}_2$.
  Applying Proposition \ref{prop:local-seed} and using
  $\sqrt\eps/\ell_n=\sqrt{\delta_n}\le\sqrt{\delta_+}$ proves the
  claim.
\end{proof}

\section{An approximate inverse for the horizontal problem on $\T^2$ and an unstable eigenvalue}
\label{sec:parametrix}

In this section, we use the instability of the
operator $A^{(n)}_{\eps,n}$ on $\R^2$ for the selected cell to prove that the horizontal
operator $A_{\eps,K_n}$ on $\T^2$ has an eigenvalue with positive
real part. We first
combine the resolvent on $\R^2$ for the selected cell and the background resolvent on $\T^2$ by spatial cut-offs. We then test the contour integral against the fixed
function from Proposition \ref{prop:local-seed}.

Similar techniques were also used in the context of complex Schr\"odinger operators in
\cite[Lemma~2]{Bogli17} and \cite[Section~5.2, Proposition~18]{Cuenin22}.

\begin{proposition}[Existence of an unstable eigenvalue of the horizontal operator on $\T^2$]
  \label{prop:global-spectral-transfer}
  Fix $v,I,\Gamma,\cG,g_*,f_*,L_*$, and choose $q$, the axial cut-off
  $\vartheta$ and $\tau$ in the order specified above. Then $K_0$ can
  be chosen so that, for every
  $0<\eps\le\eps_0$ and every $n$ satisfying
  $\delta_n=\eps/\ell_n^2\in I$, we have
  \[
    \Gamma_\tau\subset\rho(A_{\eps,K_n}),
    \qquad
    Q_{\eps,n}:=\frac1{2\pi i}\int_{\Gamma_\tau}
    (z-A_{\eps,K_n})^{-1}\dd z\ne0.
  \]
  In particular, $A_{\eps,K_n}$ has an eigenvalue inside
  $\Gamma_\tau$ with real part at least $g_*$.
\end{proposition}

\subsection{The two-resolvent construction}

We combine the resolvent on $\R^2$ for the selected cell and the background resolvent on $\T^2$ by spatial cut-offs; compare
\cite[Section~3]{DatchevVasy12}. The cut-offs are non-constant only where the horizontal
transport vanishes. Thus only diffusion commutators remain.

We fix $\eps$ and $n$ as in the proposition. We choose smooth scalar cut-offs
$\phi_n$, $\psi_n$ and $\chi_n$ with values in $[0,1]$.
The \emph{input cut-off} $\phi_n$ splits a function $f$
into $\phi_n f$ for the model on $\R^2$ for the selected cell and
$(1-\phi_n)f$ for the background problem on $\T^2$. After applying the
two resolvents, we multiply the solution on $\R^2$ for the selected cell by $\psi_n$
and the background solution on $\T^2$ by $\chi_n$, the \emph{output
  cut-offs}.

We choose the cut-offs such that
\[
  \phi_n=1\text{ on }D(p_n,R_n/5),
  \qquad \supp\phi_n\subset D(p_n,R_n/4),
\]
\[
  \psi_n=1\text{ on }D(p_n,R_n/3),
  \qquad \supp\psi_n\subset D(p_n,R_n/2),
\]
and
\[
  \chi_n=0\text{ on }D(p_n,R_n/10),
  \qquad \chi_n=1\text{ outside }D(p_n,R_n/6).
\]
We choose the functions by scaling fixed radial cut-offs, so that
\begin{equation}
  \norm{D^j\psi_n}_\infty+\norm{D^j\chi_n}_\infty
  \le CR_n^{-j},
  \qquad j=1,2.
  \label{eq:parametrix-cutoff-bounds}
\end{equation}
The input and output cut-offs satisfy
\begin{equation}
  \chi_n(1-\phi_n)=1-\phi_n,
  \qquad \psi_n\phi_n=\phi_n.
  \label{eq:parametrix-cutoff-identities}
\end{equation}
Thus $\chi_n$ equals one on $\supp(1-\phi_n)$, and $\psi_n$
equals one on $\supp\phi_n$.

For the background resolvent on $\T^2$, we abbreviate
\[
  \cR_{\bg}(z)=(z-A^{\bg}_{\eps,n})^{-1},
\]
and define the approximate inverse for the horizontal problem on $\T^2$
\begin{equation}
  \cR_{\rm app}(z)
  =\chi_n\cR_{\bg}(z)(1-\phi_n)
  +\psi_n\cR_n(z)\phi_n,
  \qquad z\in\Gamma_\tau.
  \label{eq:Rapp}
\end{equation}
For the second term, we extend $\phi_nf$ by zero from the fixed chart to
$\R^2$, apply the resolvent on $\R^2$ for the selected cell, and restrict the product with
$\psi_n$ to the torus. Both cut-offs are supported inside the
chart, so this defines a bounded map from $\Lp^2(\T^2;\CC^2)$ to
$\Hs^2(\T^2;\CC^2)$. The background term has the same mapping property.
Equations \eqref{eq:bg-resolvent} and \eqref{eq:local-gradient} give
\begin{equation}   \label{eq:parametrix-uniform-bound}
  \sup_{z\in\Gamma_\tau}\norm{{\cR}_{\rm app}(z)}_{2\to2}
  \le C_{\rm app},
\end{equation}
uniformly in $n$, $\delta_n\in I$ and in the enlargement of $K_0$.

\begin{lemma}[Identity and error for the approximate inverse]
  \label{lem:parametrix-error}
  For $z\in\Gamma_\tau$, we have
  \begin{equation}
    (z-A_{\eps,K_n})\cR_{\rm app}(z)=\id+E_n(z),
    \label{eq:parametrix}
  \end{equation}
  where
  \begin{equation}
    \norm{E_n(z)}_{2\to2}
    \le C\left(\frac{\sqrt\eps}{R_n}+\frac\eps{R_n^2}\right)
    \le C(\sqrt{\ell_n}+\ell_n).
    \label{eq:explicit-error}
  \end{equation}
  The constant is independent of $n$, $\delta_n\in I$, $z$ and the
  final enlargement of $K_0$.
\end{lemma}

\begin{proof}
  We use the convention
  $[A_{\eps,K_n},\chi]=A_{\eps,K_n}\chi-\chi A_{\eps,K_n}$. The coefficients of the horizontal and background operators on $\T^2$ agree
  wherever $\chi_n\ne0$, since
  $\supp u_n^{\loc}\subset D(p_n,R_n/10)$. The coefficients of the horizontal operator on $\T^2$ and the operator
  on $\R^2$ for the selected cell agree on $\supp\psi_n$, since the added axial
  velocity profile in the selected cell of $\T^2$ is constant there and all other cells
  are disjoint.
  Therefore \eqref{eq:parametrix-cutoff-identities} gives
  \begin{align*}
    (z-A_{\eps,K_n})\chi_n\cR_{\bg}(z)(1-\phi_n)
     & =1-\phi_n-[A_{\eps,K_n},\chi_n]\cR_{\bg}(z)(1-\phi_n), \\
    (z-A_{\eps,K_n})\psi_n\cR_n(z)\phi_n
     & =\phi_n-[A_{\eps,K_n},\psi_n]\cR_n(z)\phi_n.
  \end{align*}
  This proves \eqref{eq:parametrix} with
  \[
    E_n(z)=-[A_{\eps,K_n},\chi_n]\cR_{\bg}(z)(1-\phi_n)
    -[A_{\eps,K_n},\psi_n]\cR_n(z)\phi_n.
  \]

  The horizontal velocity vanishes on the transition regions of both
  output cut-offs. Every zeroth-order coefficient commutes with a scalar
  cut-off. Hence, for $\chi\in\{\chi_n,\psi_n\}$,
  \[
    [A_{\eps,K_n},\chi]h
    =\eps\left(2\nabla_\perp\chi\cdot\nabla_\perp h
    +(\Delta_\perp\chi)h\right).
  \]
  Applying \eqref{eq:parametrix-cutoff-bounds}, we obtain
  \[
    \norm{[A_{\eps,K_n},\chi]h}_2
    \le C\left(\frac\eps{R_n}\norm{\nabla_\perp h}_2
    +\frac\eps{R_n^2}\norm h_2\right).
  \]
  The resolvent gradient bounds in \eqref{eq:bg-resolvent} and
  \eqref{eq:local-gradient} cost a factor $\eps^{-1/2}$, leaving
  the ratio $\sqrt\eps/R_n$. Cutting off on the core scale instead
  would leave $\sqrt\eps/\ell_n\asymp1$: small diffusivity alone
  would not make the error small.
  Equations \eqref{eq:bg-resolvent} and \eqref{eq:local-gradient} yield
  the first inequality in \eqref{eq:explicit-error}. The second
  follows from $\eps=\delta_n\ell_n^2$, $R_n=\sqrt{\ell_n}$ and
  $\delta_n\in I$. This proves the claim.
\end{proof}

\begin{lemma}[The resolvent of the horizontal operator on $\T^2$ along the shifted contour]
  \label{lem:global-contour-resolvent}
  For all sufficiently large $K_0$, uniformly for $n\ge0$ and
  $\eps=\delta_n\ell_n^2$ with $\delta_n\in I$, we have
  $\Gamma_\tau\subset\rho(A_{\eps,K_n})$. Moreover,
  \begin{equation}
    (z-A_{\eps,K_n})^{-1}
    =\cR_{\rm app}(z)(\id+E_n(z))^{-1},
    \qquad z\in\Gamma_\tau.
    \label{eq:global-resolvent}
  \end{equation}
\end{lemma}

\begin{proof}
  Since $\ell_n\le\ell_0=1/K_0$, the right-hand side of
  \eqref{eq:explicit-error} tends to zero uniformly as $K_0\to\infty$.
  We choose $0<\eta_*<1/2$ so that
  \[
    \frac{\operatorname{length}(\Gamma)}{2\pi}
    C_{\rm app}\frac{\eta_*}{1-\eta_*}<\frac14.
  \]
  We enlarge $K_0$ until
  \begin{equation}
    \sup_{n\ge0}\sup_{\delta\in I}\sup_{z\in\Gamma_\tau}
    \norm{E_{n,\delta}(z)}_{2\to2}<\eta_*.
    \label{eq:error-small}
  \end{equation}
  Here $E_{n,\delta}$ denotes the error with $\eps=\delta\ell_n^2$.
  The Neumann series gives $(\id+E_n(z))^{-1}$, and equation
  \eqref{eq:parametrix} supplies the right inverse in equation
  \eqref{eq:global-resolvent}.

  It remains to justify that this right inverse is the resolvent. By
  Lemma \ref{lem:torus-sector-basic}, the map
  $z-A_{\eps,K_n}\colon\Hs^2\to\Lp^2$ is Fredholm of index zero. The
  constructed right inverse takes values in $\Hs^2$ and implies
  surjectivity. Index zero then gives injectivity. This establishes \eqref{eq:global-resolvent} and we complete the proof.
\end{proof}

\subsection{The Riesz projection of the horizontal operator on $\T^2$ is non-zero}
Lemma \ref{lem:global-contour-resolvent} tells us that there is no spectrum on the shifted contour. Let us now verify that there is spectrum inside.

We test the contour integral against the function from Proposition \ref{prop:local-seed}.

\begin{lemma}[The Riesz projection of the horizontal operator on $\T^2$ is non-zero]
  \label{lem:global-riesz-nonzero}
  The Riesz projection \(Q_{\eps,n}\) of the horizontal operator on $\mathbb{T}^2$ in
  Proposition \ref{prop:global-spectral-transfer} is non-zero for every
  $n\ge0$ and $\eps=\delta_n\ell_n^2$ with $\delta_n\in I$.
\end{lemma}

\begin{proof}
  We use $f_*$ and $L_*$ from Proposition \ref{prop:local-seed}.
  Since $L_*\le R_v$, \eqref{eq:local-support-in-buffer}
  gives
  \begin{equation}     \label{eq:scaled-witness-support}
    L_*\ell_n<R_n/5 \quad\text{for every }n\ge0.
  \end{equation}
  We regard $f_n=U_nf_*$ as a function on the torus by restriction to the fixed
  chart and subsequent extension by zero. Then $\norm{f_n}_2=1$ and
  \(\phi_nf_n={\psi_n}f_n=f_n\). Applying \eqref{eq:Rapp} and \eqref{eq:scaling}, we obtain
  \begin{align*}
    \langle\cR_{\rm app}(z)f_n,f_n\rangle
     & =\langle\cR_n(z)f_n,f_n\rangle                       \\
     & =\langle(z+i\tau-\cL_{\delta_n})^{-1}f_*,f_*\rangle,
    \qquad z \in \Gamma_\tau.
  \end{align*}
  Indeed, the input of the background term vanishes, while the output
  cut-off in the term for the selected cell equals one on $\supp f_n$. Integrating
  around the shifted contour gives
  \begin{equation}
    \frac1{2\pi i}\int_{\Gamma_\tau}
    \langle\cR_{\rm app}(z)f_n,f_n\rangle\dd z
    =\langle P_{\delta_n}f_*,f_*\rangle.
    \label{eq:selected-contour-identity}
  \end{equation}
  The output $\cR_n(z)f_n$ of the resolvent on $\R^2$ for the selected cell may extend beyond that cell.
  Since $f_n$ vanishes outside the selected cell, these values do not
  contribute to $\inner{\cR_n(z)f_n}{f_n}$.

  Next, \eqref{eq:global-resolvent} yields
  \[
    (z-A_{\eps,K_n})^{-1}-\cR_{\rm app}(z)
    =-\cR_{\rm app}(z)(\id+E_n(z))^{-1}E_n(z).
  \]
  Using $\norm{f_n}_2=1$ and integrating over the contour, we deduce
  from equations \eqref{eq:parametrix-uniform-bound} and
  \eqref{eq:error-small} that
  \begin{equation}
    \abs{\langle Q_{\eps,n}f_n,f_n\rangle
      -\langle P_{\delta_n}f_*,f_*\rangle}
    \le\frac{\operatorname{length}(\Gamma)}{2\pi}
    C_{\rm app}\frac{\eta_*}{1-\eta_*}<\frac14.
    \label{eq:global-projection-error}
  \end{equation}
  Combining this estimate with \eqref{eq:local-spectral-witness},
  we deduce $\Re\langle Q_{\eps,n}f_n,f_n\rangle>1/4$ and thus
  $Q_{\eps,n}\ne0$.
\end{proof}

\begin{proof}[Proof of Proposition \ref{prop:global-spectral-transfer}]
  Lemmas \ref{lem:global-contour-resolvent} and
  \ref{lem:global-riesz-nonzero} give the resolvent along the shifted contour and the
  non-zero Riesz projection of the horizontal operator on $\T^2$. By Lemma \ref{lem:torus-sector-basic}, the
  resolvent is compact. Thus the spectrum inside \(\Gamma_\tau\) consists
  of isolated eigenvalues of finite algebraic multiplicity. The
  non-zero projection shows that this spectrum is nonempty. Every
  eigenvalue in this region has real part at least $g_*$, as
  translation by $-i\tau$ preserves real parts. This completes the proof.
\end{proof}

Let us collect the parameter choices. We first fix
$v,I,\Gamma,\cG,g_*,f_*,L_*$, and then choose
$1<q<\sqrt{\delta_+/\delta_-}$ and the axial cut-off $\vartheta$.
Next, we fix $\tau$ by equation
\eqref{eq:axial-separation-choice}. Finally, we choose $K_0$ large enough
for the bounds on $K_{n+1}/K_n$ in \eqref{eq:ratio-choice}, the packing in equation
\eqref{eq:chart-packing}, the core relations in equation
\eqref{eq:local-support-in-buffer}, and equations
\eqref{eq:background-diffusion-small} and \eqref{eq:error-small}.
The choice of $K_0$ fixes the velocity field on $\T^3$ and $\eps_0$. Only the
index $n=n(\eps)$ depends on the given diffusivity.

\begin{remark}
  \label{rem:flexible-buffer-scales}
  The resolvent and spectral-transfer arguments above also apply to
  nonincreasing radii $R_n$ satisfying
  \begin{equation}
    \sum_{n\ge0}R_n\longrightarrow0,
    \qquad\sup_{n\ge0}\frac{\ell_n}{R_n}\longrightarrow0,
    \qquad K_0 \to \infty.
    \label{eq:abstract-buffer-conditions}
  \end{equation}
  The first condition gives the packing. Writing
  $d(K_0)=\sup_n\ell_n/R_n$, we bound the diffusion error in equation
  \eqref{eq:background-coercivity} by $C d(K_0)^2$. Moreover, we bound the commutator
  error in equation \eqref{eq:explicit-error} by
  $C(d(K_0)+d(K_0)^2)$. Moreover, the second condition gives core
  containment and \eqref{eq:scaled-witness-support}.
  In particular, $R_n=\ell_n^\alpha$ works for
  every fixed $0<\alpha<1$. The choice $\alpha=1/2$ is only a convenient scale for these estimates.
\end{remark}

\section{Recovery of the magnetic field and proof of the main theorem}
\label{sec:finish}

We are now able to recover the axial magnetic component from the
eigenfunction of the horizontal operator on $\T^2$ via the divergence constraint.
Then we deduce the eigenvalue equation for the induction operator on $\T^3$ in the sense of
distributions. We use elliptic regularity to prove membership in the
operator domain. In the end, we pass from this possibly complex-valued eigenvalue and eigenfunction
to a real-valued solution and prove an estimate on the magnetic length.

\begin{lemma}[Recovery of the axial magnetic component]
  \label{lem:axial-recovery}
  Let $\eps>0$, $K\in\Z\setminus\{0\}$ and $\lambda\in\CC$.
  Let $u\in\W^{1,\infty}(\T^3;\R^3)$ be divergence-free and
  independent of $x_3$. Suppose that
  $0\neq h\in\Hs^2(\T^2;\CC^2)$ satisfies
  $A_{\eps,K}h=\lambda h$. We set
  \begin{align*}
    b_3=-\frac{\Div_\perp h}{2\pi iK},
    \qquad b=(h,b_3),
    \qquad V(y,x_3)=\e^{2\pi iKx_3}b(y).
  \end{align*}
  Then $b\in\Hs^2(\T^2;\CC^3)$ and
  \[
    0\neq V\in\Hs^2(\T^3;\CC^3)\cap\Lp^2_\sigma(\T^3),
    \qquad L_{\eps,u}V=\lambda V.
  \]
\end{lemma}

\begin{proof}
  By construction, $V\in\Hs^1$ and
  $\Div V=0$. Let us set
  \[
    F=\lambda V-\eps\Delta V-\curl(u\times V)
    \in\Hs^{-1}(\T^3;\CC^3).
  \]
  Since $u\in\W^{1,\infty}$ and both vector fields $u,V$ are divergence-free,
  the identity
  $\curl(u\times V)=(V\cdot\nabla)u-(u\cdot\nabla)V$
  holds in distributions. The vector field $u$ is independent of $x_3$ and thus the
  horizontal eigenvalue equation yields $F_\perp=0$. Moreover,
  $\Div F=0$, and the axial Fourier index of every term in $F$ is equal to $K$.
  Hence
  \[
    0=\Div F=\partial_{x_3}F_3=2\pi iKF_3
  \]
  in distributions. Since $K\neq0$, we conclude that $F=0$.

  Finally, $u\times V\in\Hs^1$, so the equation gives
  \[
    \eps\Delta V=\lambda V-\curl(u\times V)\in\Lp^2.
  \]
  Elliptic regularity yields $V\in\Hs^2$, and taking its
  $K$-th Fourier coefficient gives $b\in\Hs^2$. Thus
  $L_{\eps,u}V=\lambda V$. Since $h\neq0$, also $V\neq0$.
  This completes the proof.
\end{proof}

\begin{lemma}[The real part of an axial Fourier mode]
  \label{lem:real-axial-mode}
  Let $K\in\Z\setminus\{0\}$, $b\in\Lp^2(\T^2;\CC^3)$ and
  $\varphi\in\R$. Then, for almost every $y\in\T^2$,
  \begin{equation}
    \int_0^1
    \abs{\Re\left(\e^{i\varphi}\e^{2\pi iKx_3}b(y)\right)}^2
    \dd x_3=\frac12\abs{b(y)}^2.
    \label{eq:real-part-norm}
  \end{equation}
  In particular, this real part is non-zero whenever $b\neq0$.
\end{lemma}

\begin{proof}
  We write $b=a+ic$, where $a$ and $c$ are real-valued. Then
  \[
    \Re\left(\e^{i\varphi}\e^{2\pi iKx_3}b\right)
    =a\cos(2\pi Kx_3+\varphi)-c\sin(2\pi Kx_3+\varphi).
  \]
  Since $K$ is a non-zero integer, the squared sine and cosine have
  average $1/2$, and their product has average zero on $[0,1]$.
  Integrating the squared norm proves the claim.
\end{proof}

\begin{proof}[Proof of Theorem \ref{thm:main}]
  Let $0<\eps\le\eps_0$ and choose $n=n(\eps)$ as in Section
  \ref{sec:axial}. Proposition \ref{prop:global-spectral-transfer}
  provides an eigenvalue $\lambda_\eps=\alpha_\eps+i\omega_\eps$ and
  an eigenfunction $0\neq h^\eps\in\Hs^2(\T^2;\CC^2)$ such that
  \[
    A_{\eps,K_n}h^\eps=\lambda_\eps h^\eps,
    \qquad \alpha_\eps\ge g_*>0.
  \]
  We apply Lemma \ref{lem:axial-recovery} to obtain
  $b^\eps\in\Hs^2(\T^2;\CC^3)$ such that
  $V^\eps(y,x_3)=\e^{2\pi iK_nx_3}b^\eps(y)$ is a non-zero
  element of $\Hs^2(\T^3;\CC^3)\cap\Lp^2_\sigma(\T^3)$ which
  satisfies $L_{\eps,u}V^\eps=\lambda_\eps V^\eps$.

  We set $B_{\inn}^\eps=\Re V^\eps$. Since the coefficients of the
  induction equation are real-valued, its solution is
  \[
    B^\eps(t)=\e^{\alpha_\eps t}
    \Re\left(\e^{i\omega_\eps t}V^\eps\right).
  \]
  Applying Lemma \ref{lem:real-axial-mode} with phases
  $\omega_\eps t$ and zero, we obtain
  \[
    \norm{\Re(\e^{i\omega_\eps t}V^\eps)}_2^2
    =\frac12\norm{b^\eps}_2^2
    =\norm{B_{\inn}^\eps}_2^2>0.
  \]
  Consequently,
  \[
    \norm{B^\eps(t)}_2
    =\e^{\alpha_\eps t}\norm{B_{\inn}^\eps}_2
    \qquad\text{for every }t\ge0.
  \]
  The spectral and real-valued assertions therefore hold with
  \[
    \gamma_0=g_*,\qquad K_\eps=K_{n(\eps)}.
  \]
  The exact growth law gives
  \[
    \norm{\e^{tL_{\eps,u}}}_{\Lp^2_\sigma\to\Lp^2_\sigma}
    \ge
    \frac{\norm{B^\eps(t)}_2}{\norm{B_{\inn}^\eps}_2}
    =\e^{\alpha_\eps t}.
  \]
  Hence $\gamma(u,\eps)\ge\alpha_\eps\ge\gamma_0$ for every
  $0<\eps\le\eps_0$.
  Taking the lower limit as $\eps\downarrow0$ proves
  \eqref{eq:main-growth-conclusion}. The assertions on the particle flow follow
  have been shown in Proposition \ref{prop:particle-flow}. Moreover, the assertions on the ideal problem
  follow from Corollary \ref{cor:ideal-growth}. This completes the proof.
\end{proof}

\begin{corollary}[Axial Fourier index and magnetic length]
  \label{cor:magnetic-scale}
  Let $V^\eps$ be the magnetic eigenmode from Theorem \ref{thm:main}.
  We define the $\Hs^1$-based magnetic length by
  \[
    L_B^\eps=\frac{\norm{V^\eps}_2}{\norm{\nabla V^\eps}_2}.
  \]
  Then there are constants $c,C>0$, independent of $\eps$, such that
  \[
    c\eps^{-1/2}\le\abs{K_\eps}\le C\eps^{-1/2},
    \qquad
    c\sqrt\eps\le L_B^\eps\le C\sqrt\eps,
    \qquad 0 < \eps \le \eps_0.
  \]
  The real datum $B_{\inn}^\eps$ has the same magnetic length.
\end{corollary}

\begin{proof}
  We first recall that
  \[
    \ell_n=\sqrt{\frac{\eps}{\delta_n}},
    \qquad K_n\ell_n=1,
    \qquad\delta_n\in[\delta_-,\delta_+].
  \]
  Hence $K_\eps=\sqrt{\delta_n/\eps}$, which proves the first
  estimate. Since $\partial_{x_3}V^\eps=2\pi iK_nV^\eps$, we have
  \[
    L_B^\eps\le\frac1{2\pi K_n}
    =\frac{\ell_n}{2\pi}
    \le\frac{\sqrt\eps}{2\pi\sqrt{\delta_-}}.
  \]

  To prove the lower bound, we introduce
  \[
    M_u=\mathop{\mathrm{ess\,sup}}_{x\in\T^3}
    \norm{\nabla u(x)}_{\mathrm{op}}.
  \]
  Testing the eigenvalue equation with $V^\eps$, taking real
  parts and using incompressibility, we obtain
  \begin{align*}
    \eps\norm{\nabla {V}^\eps}_2^2
     & \le\Re\int_{\T^3}(V^\eps\cdot\nabla)u\cdot
    \overline{V^\eps}\dd x                        \\
     & \le M_u\norm{V^\eps}_2^2.
  \end{align*}
  Here we have used $\Re{}\lambda_\eps>0$. In particular,
  $M_u>0$. Consequently,
  \[
    L_B^\eps\ge\sqrt{\frac{\eps}{M_u}}.
  \]
  Finally, each spatial derivative of $V^\eps$ lies in the same non-zero
  axial Fourier mode. Applying Lemma \ref{lem:real-axial-mode} to
  $V^\eps$ and its derivatives gives
  \[
    \norm{B_{\inn}^\eps}_2^2=\frac{1}{2}\norm{V^\eps}_2^2,
    \qquad
    \norm{\nabla B_{\inn}^\eps}_2^2
    =\frac12\norm{\nabla V^\eps}_2^2.
  \]
  Thus $B_{\inn}^\eps$ has magnetic length $L_B^\eps$. This completes
  the proof.
\end{proof}

\begin{remark}
  The growing magnetic modes constructed here therefore have a
  magnetic length comparable to the classical resistive scale
  $\sqrt\eps$; see \cite[Section~3]{MoffattProctor85}.
\end{remark}

\appendix

\section{The semigroup and its spectral growth}
\label{sec:semigroup-growth}

For the readers' convenience we justify the semigroup formulation and the reduction of the growth
rate to the spectral bound used in the introduction.

\begin{lemma}[Semigroup generation and spectral growth]
  \label{lem:semigroup-spectral-growth}
  Let $u\in\W^{1,\infty}(\T^3;\R^3) $ be divergence-free and let
  $\eps>0 $. The induction operator $L_{\eps,u} $ on $\T^3$, with domain
  $\Hs^2\cap\Lp^2_\sigma $, generates an analytic $\C_0 $-semigroup on
  $\Lp^2_\sigma $.

  The resolvent of $L_{\eps,u}$ is compact,
  $\e^{tL_{\eps,u}} $ is compact for every $t>0 $, and
  \begin{equation}
    \sigma(\e^{tL_{\eps,u}}){}\setminus\{0\}  \label{eq:compact-semigroup-spectral-mapping}
    =
    \left\{\e^{t\lambda}\, : \,\lambda\in\sigma(L_{\eps,u})\right\},
    \qquad t>0.
  \end{equation}
  Furthermore,
  \begin{equation}
    \gamma(u,\eps) \label{eq:growth-rate-spectral-bound}
    =
    \max\{\Re\lambda\, : \,\lambda\in\sigma(L_{\eps,u})\}.
  \end{equation}
\end{lemma}

\begin{proof}
  We work on $\Lp^2_\sigma$. The operator $\eps\Delta$, with
  domain $\Hs^2\cap\Lp^2_\sigma$, generates an analytic semigroup
  on this space. This can be seen by looking at its Fourier representation. Moreover,
  $C_uB=\curl(u\times B)$ belongs to $\Lp^2_\sigma$ for every
  $B\in\Hs^2\cap\Lp^2_\sigma$. For every $\eta>0$, interpolation shows
  \[
    \norm{C_uB}_2
    \le C\norm{ u }_{\W^{1,\infty}}\norm B_{\Hs^1}
    \le\eta\norm{\Delta B}_2+C_\eta{}\norm B_2.
  \]
  Thus $C_u$ is infinitesimally bounded relative to $\eps\Delta$
  on $\Lp^2_\sigma$. Applying
  \cite[Theorem~III.2.10]{EngelNagel}, we deduce that $L_{\eps,u}$
  generates an analytic semigroup.

  By the preceding relative bound and the elliptic estimate
  $\norm B_{\Hs^2}\le C(\norm{\Delta B}_2+\norm B_2) $, the graph norm of
  $L_{\eps,u} $ is equivalent to the $\Hs^2 $-norm.
  Since $\Hs^2(\T^3)\hookrightarrow\Lp^2(\T^3) $ is compact,
  $L_{\eps,u} $ has compact resolvent. Analyticity and
  \cite[Theorem~II.4.29]{EngelNagel} show that the semigroup is compact. Then, \cite[Corollaries~IV.3.11 and~IV.3.12]{EngelNagel} give spectral mapping and
  equality of the growth and spectral bounds. The limit in
  \eqref{eq:semigroup-growth-rate} is this growth bound, which proves
  \eqref{eq:growth-rate-spectral-bound}. It remains to prove
  that the supremum is attained. For every $B\in D(L_{\eps,u})$, we have
  \[
    \int_{\T^3}L_{\eps,u}B\dd x=0.
  \]
  Since $\Lp^2_\sigma$ contains non-zero constant fields,
  $L_{\eps,u}$ is not surjective. Thus $0\in\sigma(L_{\eps,u})$.
  By \cite[Corollary~V.3.2]{EngelNagel}, only finitely many spectral
  points lie in any fixed right half-plane. Applying this to the
  half-plane $\{\Re z\ge0\}$, we conclude that the
  supremum is attained.
\end{proof}

\vfill


\begin{thebibliography}{99}

  \bibitem{ACM14}
  G. Alberti, G. Crippa and A.~L. Mazzucato,
  \emph{Exponential self-similar mixing and loss of regularity for
    continuity equations},
  C. R. Math. Acad. Sci. Paris \textbf{352} (2014), no.~11, 901--906,
  \href{https://doi.org/10.1016/j.crma.2014.08.021}
  {doi:10.1016/j.crma.2014.08.021}.

  \bibitem{ACM19}
  G. Alberti, G. Crippa and A.~L. Mazzucato,
  \emph{Loss of regularity for the continuity equation with non-Lipschitz
    velocity field},
  Ann. PDE \textbf{5} (2019), no.~1, Paper No.~9, 19~pp.,
  \href{https://doi.org/10.1007/s40818-019-0066-3}
  {doi:10.1007/s40818-019-0066-3}.

  \bibitem{AH15}
  Y. Almog and B. Helffer,
  \emph{On the spectrum of non-selfadjoint Schr\"odinger operators with
    compact resolvent},
  Comm. Partial Differential Equations \textbf{40} (2015), no.~8, 1441--1466,
  \href{https://doi.org/10.1080/03605302.2015.1025978}
  {doi:10.1080/03605302.2015.1025978}.

  \bibitem{AmbrosioCrippa14}
  L. Ambrosio and G. Crippa,
  \emph{Continuity equations and ODE flows with non-smooth velocity},
  Proc. Roy. Soc. Edinburgh Sect. A \textbf{144} (2014), no.~6, 1191--1244,
  \href{https://doi.org/10.1017/S0308210513000085}
  {doi:10.1017/S0308210513000085}.

  \bibitem{ArmstrongVicol25}
  S. Armstrong and V. Vicol,
  \emph{Anomalous diffusion by fractal homogenization},
  Ann. PDE \textbf{11} (2025), no.~1, Paper No.~2,
  \href{https://doi.org/10.1007/s40818-024-00189-6}
  {doi:10.1007/s40818-024-00189-6}.

  \bibitem{Arnold04}
  V.~I. Arnold (ed.),
  \emph{Arnold's Problems},
  Springer-Verlag, Berlin, and PHASIS, Moscow, 2004,
  Problem~1994--28,
  \href{https://doi.org/10.1007/b138219}{doi:10.1007/b138219}.

  \bibitem{AK98}
  V.~I. Arnold and B.~A. Khesin,
  \emph{Topological Methods in Hydrodynamics},
  Applied Mathematical Sciences, vol.~125, Springer, 1998,
  \href{https://doi.org/10.1007/b97593}{doi:10.1007/b97593}.

  \bibitem{AZRS81}
  V.~I. Arnol'd, Ya.~B. Zel'dovich, A.~A. Ruzmaikin and D.~D. Sokolov,
  \emph{A magnetic field in a stationary flow with stretching in Riemannian
    space},
  Sov. Phys. JETP \textbf{54} (1981), no.~6, 1083--1086,
  \href{https://jetp.ras.ru/cgi-bin/dn/e_054_06_1083.pdf}
  {English translation}.

  \bibitem{BaylyChildress88}
  B.~J. Bayly and S. Childress,
  \emph{Construction of fast dynamos using unsteady flows and maps in three
    dimensions},
  Geophys. Astrophys. Fluid Dyn. \textbf{44} (1988), no.~1--4, 211--240,
  \href{https://doi.org/10.1080/03091928808208887}
  {doi:10.1080/03091928808208887}.

  \bibitem{Bogli17}
  S. B\"ogli,
  \emph{Schr\"odinger operator with non-zero accumulation points of
    complex eigenvalues},
  Comm. Math. Phys. \textbf{352} (2017), no.~2, 629--639,
  \href{https://doi.org/10.1007/s00220-016-2806-5}
  {doi:10.1007/s00220-016-2806-5}.

  \bibitem{CL97}
  C. Chicone and Y. Latushkin,
  \emph{The geodesic flow generates a fast dynamo: an elementary proof},
  Proc. Amer. Math. Soc. \textbf{125} (1997), no.~11, 3391--3396,
  \href{https://doi.org/10.1090/S0002-9939-97-04187-7}
  {doi:10.1090/S0002-9939-97-04187-7}.

  \bibitem{ChildressGilbert95}
  S. Childress and A.~D. Gilbert,
  \emph{Stretch, Twist, Fold: The Fast Dynamo},
  Lecture Notes in Physics Monographs, vol.~37, Springer, 1995,
  \href{https://doi.org/10.1007/978-3-540-44778-8}
  {doi:10.1007/978-3-540-44778-8}.

  \bibitem{CNF24}
  M. Coti Zelati and V. Navarro-Fern\'andez,
  \emph{Three-dimensional exponential mixing and ideal kinematic dynamo
    with randomized ABC flows},
  J. Dynam. Differential Equations \textbf{38} (2026), suppl.~1, 169--199,
  \href{https://doi.org/10.1007/s10884-026-10483-5}
  {doi:10.1007/s10884-026-10483-5}.

  \bibitem{CSV25}
  M. Coti Zelati, M. Sorella and D. Villringer,
  \emph{Alpha-unstable flows and the fast dynamo problem},
  \href{https://arxiv.org/abs/2504.00855v1}{arXiv:2504.00855v1}, 2025.

  \bibitem{CSV26}
  M. Coti Zelati, M. Sorella and D. Villringer,
  \emph{A fast dynamo on the three-torus},
  \href{https://arxiv.org/abs/2603.09861v2}{arXiv:2603.09861v2}, 2026.

  \bibitem{CSV26autonomous}
  M. Coti Zelati, M. Sorella and D. Villringer,
  \emph{Smooth autonomous fast dynamo action on the three-torus},
  \href{https://arxiv.org/abs/2609.04153v1}{arXiv:2609.04153v1}, 2026.

  \bibitem{Cuenin22}
  J.-C. Cuenin,
  \emph{Schr\"odinger operators with complex sparse potentials},
  Comm. Math. Phys. \textbf{392} (2022), no.~3, 951--992,
  \href{https://doi.org/10.1007/s00220-022-04358-1}
  {doi:10.1007/s00220-022-04358-1}.

  \bibitem{DatchevVasy12}
  K. Datchev and A. Vasy,
  \emph{Gluing semiclassical resolvent estimates via propagation of
    singularities},
  Int. Math. Res. Not. IMRN (2012), no.~23, 5409--5443,
  \href{https://doi.org/10.1093/imrn/rnr255}
  {doi:10.1093/imrn/rnr255}.

  \bibitem{DiPernaLions89}
  R.~J. DiPerna and P.-L. Lions,
  \emph{Ordinary differential equations, transport theory and Sobolev spaces},
  Invent. Math. \textbf{98} (1989), no.~3, 511--547,
  \href{https://doi.org/10.1007/BF01393835}
  {doi:10.1007/BF01393835}.

  \bibitem{EngelNagel}
  K.-J. Engel and R. Nagel,
  \emph{One-Parameter Semigroups for Linear Evolution Equations},
  Graduate Texts in Mathematics, vol.~194, Springer, 2000,
  \href{https://doi.org/10.1007/b97696}{doi:10.1007/b97696}.

  \bibitem{FriedlanderVishik91}
  S. Friedlander and M.~M. Vishik,
  \emph{Dynamo theory, vorticity generation, and exponential stretching},
  Chaos \textbf{1} (1991), no.~2, 198--205,
  \href{https://doi.org/10.1063/1.165829}
  {doi:10.1063/1.165829}.

  \bibitem{GKS18}
  V. Galitski, M. Kargarian and S. Syzranov,
  \emph{Dynamo effect and turbulence in hydrodynamic Weyl metals},
  Phys. Rev. Lett. \textbf{121} (2018), no.~17, Paper No.~176603,
  \href{https://doi.org/10.1103/PhysRevLett.121.176603}
  {doi:10.1103/PhysRevLett.121.176603}.

  \bibitem{GVR07}
  D. G\'erard-Varet and F. Rousset,
  \emph{Shear layer solutions of incompressible MHD and dynamo effect},
  Ann. Inst. H. Poincar\'e C Anal. Non Lin\'eaire
  \textbf{24} (2007), no.~5, 677--710,
  \href{https://doi.org/10.1016/j.anihpc.2006.04.005}
  {doi:10.1016/j.anihpc.2006.04.005}.

  \bibitem{Gilbert88}
  A.~D. Gilbert,
  \emph{Fast dynamo action in the Ponomarenko dynamo},
  Geophys. Astrophys. Fluid Dyn. \textbf{44} (1988), no.~1--4, 241--258,
  \href{https://doi.org/10.1080/03091928808208888}
  {doi:10.1080/03091928808208888}.

  \bibitem{GGK90}
  I. Gohberg, S. Goldberg and M.~A. Kaashoek,
  \emph{Classes of Linear Operators, Vol.~I},
  Operator Theory: Advances and Applications, vol.~49,
  Birkh\"auser, Basel, 1990,
  \href{https://doi.org/10.1007/978-3-0348-7509-7}
  {doi:10.1007/978-3-0348-7509-7}.

  \bibitem{HasselblattKatok02}
  B. Hasselblatt and A. Katok,
  \emph{Principal structures},
  in \emph{Handbook of Dynamical Systems}, vol.~1A,
  B. Hasselblatt and A. Katok (eds.), Elsevier, Amsterdam, 2002, 1--203,
  \href{https://doi.org/10.1016/S1874-575X(02)80003-0}
  {doi:10.1016/S1874-575X(02)80003-0}.

  \bibitem{Jabin16}
  P.-E. Jabin,
  \emph{Critical non-Sobolev regularity for continuity equations with
    rough velocity fields},
  J. Differential Equations \textbf{260} (2016), no.~5, 4739--4757,
  \href{https://doi.org/10.1016/j.jde.2015.11.028}
  {doi:10.1016/j.jde.2015.11.028}.

  \bibitem{Kato}
  T. Kato,
  \emph{Perturbation Theory for Linear Operators},
  Classics in Mathematics, Springer, 1995,
  \href{https://doi.org/10.1007/978-3-642-66282-9}
  {doi:10.1007/978-3-642-66282-9}.

  \bibitem{KlapperYoung95}
  I. Klapper and L.~S. Young,
  \emph{Rigorous bounds on the fast dynamo growth rate involving
    topological entropy},
  Comm. Math. Phys. \textbf{173} (1995), no.~3, 623--646,
  \href{https://doi.org/10.1007/BF02101659}
  {doi:10.1007/BF02101659}.

  \bibitem{KRRS17}
  D. Krej\v{c}i\v{r}\'ik, N. Raymond, J. Royer and P. Siegl,
  \emph{Non-accretive Schr\"odinger operators and exponential decay of
    their eigenfunctions},
  Israel J. Math. \textbf{221} (2017), 779--802,
  \href{https://doi.org/10.1007/s11856-017-1574-z}
  {doi:10.1007/s11856-017-1574-z}.

  \bibitem{MoffattProctor85}
  H.~K. Moffatt and M.~R.~E. Proctor,
  \emph{Topological constraints associated with fast dynamo action},
  J. Fluid Mech. \textbf{154} (1985), 493--507,
  \href{https://doi.org/10.1017/S002211208500163X}
  {doi:10.1017/S002211208500163X}.

  \bibitem{NF26}
  V. Navarro-Fern\'andez,
  \emph{Exponential growth and decay in the ideal induction equation},
  arXiv preprint arXiv:2608.05997 (2026),
  \href{https://arxiv.org/abs/2608.05997}{arXiv:2608.05997}.

  \bibitem{NV25}
  V. Navarro-Fern\'andez and D. Villringer,
  \emph{Spectral instability in the smooth Ponomarenko dynamo},
  Ann. PDE \textbf{12} (2026), no.~2, Paper No.~32,
  \href{https://doi.org/10.1007/s40818-026-00255-1}
  {doi:10.1007/s40818-026-00255-1}.

  \bibitem{NiebelLean}
  L. Niebel,
  \emph{FastDynamo: a Lean 4 formalisation of `A fast dynamo with zero topological entropy'},
  GitHub repository,
  \url{https://github.com/lukasniebel/fast_dynamo_lean}.

  \bibitem{DLMF}
  NIST Digital Library of Mathematical Functions,
  \emph{Bessel functions},
  \url{https://dlmf.nist.gov/10}, accessed 15 September 2026.

  \bibitem{Ponomarenko73}
  Y.~B. Ponomarenko,
  \emph{Theory of the hydromagnetic generator},
  J. Appl. Mech. Tech. Phys. \textbf{14} (1973), 775--778,
  \href{https://doi.org/10.1007/BF00853190}
  {doi:10.1007/BF00853190}.

  \bibitem{Rowan25}
  K. Rowan,
  \emph{A subsequentially fast dynamo on $\T^3$},
  Arch. Ration. Mech. Anal. \textbf{250} (2026), no.~6, Paper No.~100,
  \href{https://doi.org/10.1007/s00205-026-02255-8}
  {doi:10.1007/s00205-026-02255-8}.

  \bibitem{Rowan26random}
  K. Rowan,
  \emph{An AI-discovered smooth random fast dynamo on $\T^3$},
  \href{https://arxiv.org/abs/2608.20105v1}{arXiv:2608.20105v1}, 2026.

  \bibitem{RSS88}
  A.~A. Ruzmaikin, D.~D. Sokoloff and A.~M. Shukurov,
  \emph{Hydromagnetic screw dynamo},
  J. Fluid Mech. \textbf{197} (1988), 39--56,
  \href{https://doi.org/10.1017/S0022112088003167}
  {doi:10.1017/S0022112088003167}.

  \bibitem{SV25}
  M. Sorella and D. Villringer,
  \emph{A limsup fast dynamo on $\T^3 $},
  \href{https://arxiv.org/abs/2511.23024v2}{arXiv:2511.23024v2}, 2025.

  \bibitem{VZ72}
  S.~I. Vainshtein and Ya.~B. Zel'dovich,
  \emph{Origin of magnetic fields in astrophysics
    (turbulent ``dynamo'' mechanisms)},
  Sov. Phys. Usp. \textbf{15} (1972), no.~2, 159--172,
  \href{https://doi.org/10.1070/PU1972v015n02ABEH004960}
  {doi:10.1070/PU1972v015n02ABEH004960}.

  \bibitem{Vishik89}
  M.~M. Vishik,
  \emph{Magnetic field generation by the motion of a highly conducting fluid},
  Geophys. Astrophys. Fluid Dyn. \textbf{48} (1989), nos.~1--3, 151--167,
  \href{https://doi.org/10.1080/03091928908219531}
  {doi:10.1080/03091928908219531}.

\end{thebibliography}
\end{document}